\documentclass[a4paper,UKenglish,cleveref,numberwithinsect,nolineno]{socg-lipics-v2021}

\pdfoutput=1 
\hideLIPIcs  

\title{Estimation of Conformal Metrics}

\author{Jérôme Taupin}{Université Paris-Saclay, France \and INRIA Saclay, France}{}{}{}

\authorrunning{J. Taupin}



\usepackage{tikz,tikz-3dplot}
\usetikzlibrary{patterns}

\usepackage{acronym}

\usepackage{macros} 

\newcommand{\fmin}{f_{\min}}
\newcommand{\metricratioOld}[2]{\Lv 1-\frac{#1}{#2} \Rv}
\newcommand{\metricratio}[2]{\Lv \frac{#1-#2}{#1\vee#2} \Rv}
\newcommand{\damb}{{N}}

\EventEditors{}
\EventNoEds{2}
\EventLongTitle{}
\EventShortTitle{}
\EventAcronym{}
\EventYear{}
\EventDate{}
\EventLocation{}
\EventLogo{}
\SeriesVolume{42}
\ArticleNo{23}

\begin{document}

\maketitle

\begin{abstract}
We study deformations of the geodesic distances on a domain of $\RR^\damb$ induced by a function called conformal factor.
We show that under a positive reach assumption on the domain (not necessarily a submanifold) and mild assumptions on the conformal factor, geodesics for the conformal metric have good regularity properties in the form of a lower bounded reach. This regularity allows for efficient estimation of the conformal metric from a random point cloud with a relative error proportional to the Hausdorff distance between the point cloud and the original domain.
We then establish convergence rates of order $n^{-1/d}$ that are close to sharp when the intrinsic dimension $d$ of the domain is large, for an estimator that can be computed in $\cO(n^2)$ time.
Finally, this paper includes a useful equivalence result between ball graphs and nearest-neighbors graphs when assuming Ahlfors regularity of the sampling measure, allowing to transpose results from one setting to another.
\end{abstract}

\section{Introduction}
This paper studies metrics over subsets of the Euclidean space $(\RR^\damb,\|\cdot\|)$ obtained by conformal deformation of the shortest-path metric via a positive function. We are particularly interested in the regularity of such metrics and in their estimation via \iid sampling of points.

\begin{definition}
\label{def:conformal_metric}
Let $M \subset \RR^\damb$ be a closed path-connected \emph{domain} and $f:M\to\RR_+^*$ be a \emph{conformal factor}.
For all $x,y\in M$ the \emph{conformal distance} between $x$ and $y$ over $M$ via $f$ is defined as
\begin{align}
    \label{eq:def:conformal_metric}
    D_{M,f}(x,y)
    \eqdef \inf_{\gamma\in\Gamma_M(x,y)} \int_I f(\gamma(t)) \|\dot\gamma(t)\| \d t
\end{align}
where $\Gamma_M(x,y)$ is the set of all Lipschitz paths $\gamma:I\to M$ where $I = [a,b]$ is a nontrivial segment, $\gamma(a)=x$ and $\gamma(b)=y$.
The quantity minimized over all paths $\gamma$ is denoted
\begin{align*}
    |\gamma|_f
    \eqdef \int_I f(\gamma(t)) \|\dot\gamma(t)\| \d t
\end{align*}
and referred to as the conformal length of the curve $\gamma$ via $f$.
In the case where $f=1$, we write $D_M(x,y) = D_{M,1}(x,y)$ and $|\gamma| = \int_I \|\dot\gamma\|$, respectively the distance between $x$ and $y$ induced by the ambient metric over $M$ and the Euclidean length of a curve $\gamma$.
\end{definition}

In the following, to ensure regularity of the conformal metric, the domain $M$ is assumed to have positive reach $\tau_M$, which we recall is defined as the supremum of all $r>0$ such that any point in the offset $M^r = \{x\in\RR^\damb : d(x,M) < r\}$ has a unique projection onto~$M$---where $d(x,M) = \inf_{y\in M}\|x-y\|$ denotes the distance from point $x$ to subset~$M$---see \cite{federerCurvatureMeasures1959}.
Moreover, the conformal factor $f$ is assumed to be $\kappa$-Lipschitz and lower bounded by $\fmin>0$.
These are the sole assumptions used in this paper regarding $M$ and $f$.

Notice that $f$ is defined over $M$ in \cref{def:conformal_metric} as paths are constrained to the domain. Since any Lipschitz function defined over a subset of $\RR^\damb$ can be extended to the whole space preserving its Lipschitz property and lower bound (see \cite[Theorem 1]{mcshaneExtensionRangeFunctions1934} for instance), one may assume without loss of generality that $f$ is defined over the entire space~$\RR^\damb$, which is useful to estimate the conformal metric efficiently.
The connectedness and positive reach of $M$ ensure that for any endpoints $x,y\in M$, there exists a Lipschitz path from $x$ to $y$ in $M$, so that $\Gamma_M(x,y)$ is always nonempty and $D_{M,f}(x,y)$ is always finite. This can be deduced from \cref{lemma:metric_distortion_reach_characterization} below.
Most of the time, a path $\gamma$ is chosen to be parameterized with constant velocity, \ie, $\|\dot\gamma\|$ is constant over $I$, with $I$ being either $[0,1]$ or $[0,|\gamma|]$. The latter is referred to as an \emph{arc-length} parameterization, where $\|\dot\gamma\| = 1$ almost everywhere.

The \emph{induced metric} $D_M = D_{M,1}$ is the metric induced by the ambient space $\RR^\damb$ onto $M$.
If $M$ is a $\cC^1$ submanifold of $\RR^\damb$ and $f$ is $\cC^1$, $D_{M,f}$ is a Riemannian metric with tensor $f(x)^2 \cdot g_x$ at point $x\in M$ where $g_x$ is the tensor of the induced metric $D_M$ at $x$. The term ``conformal'' used in this paper is borrowed from the Riemannian literature.
Finally, if $\mu$ is a measure over a submanifold $M$ with density $\rho$ \ac{wrt} the volume form of $M$ and $f$ is a negative power of the density~$\rho$, $D_{M,f}$ has already been studied and is sometimes referred to as the \emph{Fermat distance} \cite{groismanNonhomogeneousEuclideanFirstpassage2022}.
This particular case is one of the main motivations for this article as the Fermat distance has been used for various practical applications, see for instance \cite{garciatrillosFermatDistancesMetric2024} and the references therein.
Particular examples of the conformal factor $f$ are discussed in \cref{sec:particular_f}.

\bigskip
Conformal metrics are included in the more general class of length-spaces, for which it is known that the infimum in \cref{def:conformal_metric} is in fact a minimum \cite[Theorem 2.5.23]{buragoCourseMetricGeometry2001}. That is, for all $x,y\in M$ there exists a path $\gamma\in\Gamma_M(x,y)$ such that $D_{M,f}(x,y) = |\gamma|_f$, called a \emph{minimizing geodesic}---and shortened to geodesic in this paper for brevity. Denote $\Gamma^\star_{M,f}(x,y)$ the set of such geodesics between two endpoints $x$ and $y$ along with $\Gamma^\star_{M,f} = \bigcup_{x\ne y\in M}\Gamma^\star_{M,f}(x,y)$ the set of all geodesics \ac{wrt} the conformal metric.
We discuss in \cref{sec:regularity} the regularity of geodesics \ac{wrt} $D_{M,f}$ and show in \cref{prop:geodesic_reach} that under the aforementioned assumptions on $M$ and $f$, geodesics have positive reach that is lower bounded by an explicit constant depending on the reach of the domain and on the conformal factor. In particular, any geodesic may be parameterized as a $\cC^{1,1}$ curve with an explicit upper bound on the Lipschitz constant of its first derivative.

We then show in \cref{sec:polygonal_approximation} that the conformal metric can be approached using polygonal paths on a weighted graph built from a point cloud $X\subset M$, provided that the graph only contains edges of length at most some threshold $r$ and that $X$ is close to $M$ in Hausdorff distance. The small edges condition ensures that paths on the graph cannot venture too far outside the domain.
This kind of construction is the same as the one used for the Isomap algorithm \cite{bernsteinGraphApproximationsGeodesics2000}, although we allow more generality by adapting the weights to the conformal factor $f$. When $f=1$ and $M$ is a geodesically convex $\cC^2$ submanifold, \cite{bernsteinGraphApproximationsGeodesics2000} provides a relative error bound of order $\frac{r^2}{\tau^2} + \frac{\rho}{r}$ where $\tau$ is the minimum radius of curvature of $M$ and $\rho$ denotes the Hausdorff distance between $M$ and $X$.
The term depending on $\rho$ can be made quadratic as shown in \cite{aaronConvergenceRatesEstimators2018}, where an assumption called geodesic smoothness that is slightly weaker than that of a positive reach allows to obtain a bound of order $Ar + \frac{\rho^2}{r^2}$ where $A$ depends on the assumption.
This assumption however lacks precision when comparing distances at a local scale which results in the first term being of order $r$ instead of $r^2$.
By using a positive reach assumption instead, the upper bound can be further improved to $\frac{r^2}{\tau^2} + \frac{\rho^2}{r^2}$ according to \cite{arias-castroPerturbationBoundsProcrustes2020} where $\tau$ is the reach of the domain, although the setup also implicitly assumes that the domain $M$ is a smooth manifold isometric to a convex domain.
Going back to our setup, we adapt these prior works to fit any conformal change. Under some condition on the weight function used, that in particular always holds for the induced metric, we establish in \cref{thm:geodesic_approx} the same upper bound on the approximation error as in \cite{arias-castroPerturbationBoundsProcrustes2020}, that is $\frac{r^2}{\tau^2} + \frac{\rho^2}{r^2}$ where $\tau$ is the explicit lower bound on the reach of any geodesic provided by \cref{prop:geodesic_reach}. This result emphasizes the fact that the only assumption on the domain needed for this level of precision is that of a positive reach. 
The approximation error can thus be made proportional to the Hausdorff distance between the point cloud and the domain by choosing appropriately $r \asymp \sqrt{\tau~\rho}$.
When the conformal factor $f$ is unknown, replacing it with an estimate $g$ does not alter the results provided that $g$ is close enough to $f$, as shown in \cref{lemma:estimator_convergence_both_errors}.

Assuming that $X = X_n$ is the outcome of $n\ge1$ \iid samples from a $d$-standard measure~$\mu$ over $M$, we study in \cref{sec:random_samples} the estimator of $D_{M,f}$ built on $X_n$ following the construction of \cref{sec:polygonal_approximation}. This estimator is shown in \cref{thm:convergence_ball_graph} to converge to $D_{M,f}$ at a rate of $n^{-1/d}$ provided that $f$ is known or can be estimated with same rate. This rate follows from \cref{thm:geodesic_approx} and the fact that $X_n$ is known to converge to $M$ in Hausdorff distance at a rate of $n^{-1/d}$.
In particular, this allows for efficient estimation of the induced metric with the sole assumption that $M$ has positive reach.
This convergence is shown using ball graphs. However, practical estimation is made easier by using nearest-neighbors graphs instead, in part because it does not require to know the intrinsic dimension $d$ to obtain an optimal convergence speed.
For this reason, we show in \cref{thm:convergence_NN_graph} that it is possible to retrieve the same convergence speed when replacing the ball graph in the estimator with a $k$-nearest-neighbors graph with $k \asymp \sqrt n$. To do so, we establish an equivalence with high probability between nearest-neighbors graphs and ball graphs when their respective parameters are properly scaled and the underlying measure is $d$-Ahlfors, see \cref{prop:graph_ball_to_NN}.
Finally, the nearest-neighbors estimator of the conformal distance between two fixed points may be computed in $\cO(n^2)$ if the time complexity of evaluating~$f$ is considered constant.
Under the stronger assumption that $M$ is a $\cC^k$ submanifold of dimension~$d$ with $k\ge2$, the minimax optimal convergence rate for the induced metric is known to be of order $n^{-k/d}$ \cite{aamariOptimalReachEstimation2023}.
When only assuming positive reach and no differential structure on $M$, the discussion of the optimal minimax rate proves to be harder as we are not able to match the upper bound of $n^{-1/d}$ and obtain a lower bound of order $n^{-1/(d-1/2)}$ instead, see \cref{thm:minimax_lower_bound}.
We discuss with more details the comparison between our setup and the $\cC^k$ case in \cref{sec:k_d}.

\section{Regularity of the Geodesics for the Conformal Metric}
\label{sec:regularity}

Recall that the term ``geodesic'' refers in this paper to a curve that achieves a global minimum of the conformal length.
In this section we study the regularity of the geodesics \ac{wrt} $D_{M,f}$.
We show that geodesics are $\cC^{1,1}$ curve with reach bounded from below by a constant depending only on $\tau_M$, $\kappa$ and~$\fmin$.
This regularity property is crucial to obtain a good approximation of the geodesics by polygonal paths.
The reach of a set can be characterized through the local distortion of the induced metric compared to the Euclidean metric.
\begin{lemma}{\cite[Theorem 1]{boissonnatReachMetricDistortion2019}}
\label{lemma:metric_distortion_reach_characterization}
The reach of $M$ may be expressed as
\begin{equation*}
    \tau_M
    = \sup\La r>0 ~:~ \forall x,y\in M,~ \|x-y\|<2r \Rightarrow D_M(x,y) \le 2r \arcsin\Lp\frac{\|x-y\|}{2r}\Rp\Ra~.
\end{equation*}
\end{lemma}

Given a path $\gamma \in \Gamma_M$ parameterized over an interval $I$ and without self-intersection, denote $\tau_\gamma$ the reach of the curve $\gamma(I) \subset \RR^\damb$ along with $D_\gamma = D_{\gamma(I)}$ the induced metric over the curve for short. $D_\gamma(x,y)$ is nothing more than the length of the curve $\gamma(I)$ between two intermediate point $x$ and $y$. In particular if $x$ and $y$ are the endpoints of $\gamma$ then $D_\gamma(x,y) = |\gamma|$. 
One consequence of \cref{lemma:metric_distortion_reach_characterization} is that the length of a geodesic between two points at most $2\tau_M$ apart is upper bounded by the length of an arc of radius $\tau_M$ between both points. That is, if $x,y\in M$ and $\|x-y\|<2\tau_M$ then
\begin{equation}
    \label{eq:arcsin_inequality}
    D_M(x,y)
    \le 2\tau_M \arcsin\Lp\frac{\|x-y\|}{2\tau_M}\Rp~.
\end{equation}
Moreover, if $\gamma$ is a geodesic, it then induces a geodesic between any of the points it goes through, which implies that $D_\gamma$ is exactly the restriction of $D_M$ to the curve $\gamma$. Then, another consequence of the characterization of \cref{lemma:metric_distortion_reach_characterization} is that the reach of $M$ is the minimal reach of any geodesic, that is
\begin{equation*}
    \tau_M
    = \inf_{\gamma\in\Gamma^\star_M} \tau_\gamma~.
\end{equation*}
We now introduce a notion of reach associated with the conformal deformation of $M$ by~$f$ using the same point of view of the metric and its geodesics.
\begin{definition}
The \emph{conformal reach} of $M$ via $f$ is defined as the minimal reach of any geodesic \ac{wrt} the conformal metric, that is
\begin{equation*}
    \tau_{M,f}
    \eqdef \inf_{\gamma\in\Gamma^\star_{M,f}} \tau_\gamma~.
\end{equation*}
\end{definition}
In the case of the induced metric this notion coincides with the usual notion of reach, that is $\tau_{M,1} = \tau_M$.
The characterization given by \cref{lemma:metric_distortion_reach_characterization} also holds for the conformal reach.
\begin{proposition}
\label{prop:metric_distortion_reach_characterization_conformal}
The conformal reach of $M$ via $f$ may be expressed as
\begin{equation}
    \label{eq:metric_distortion_reach_characterization_conformal}
    \tau_{M,f}
    = \sup\La
    \begin{gathered}
        r>0 ~:~ \forall x,y\in M, \forall \gamma\in\Gamma^\star_{M,f}(x,y),\\
        \|x-y\|<2r \Rightarrow |\gamma| \le 2r \arcsin\Lp\frac{\|x-y\|}{2r}\Rp
    \end{gathered}
    \Ra~.
\end{equation}
\end{proposition}

Beware that \cref{eq:metric_distortion_reach_characterization_conformal} involves geodesic \ac{wrt} the conformal metric $D_{M,f}$, but compares their Euclidean length---not conformal---to the one of an arc of radius $r$.

\begin{proof}
Denote $A$ the subset of $\RR_+^*$ that appears in the right-hand side of \cref{eq:metric_distortion_reach_characterization_conformal} and let us show that its supremum is indeed $\tau_{M,f}$. Since a geodesic induces a geodesic between any pair of points it goes through, one can write $A = \bigcap_\gamma A_\gamma$ where the intersection is taken over all geodesics \ac{wrt} $D_{M,f}$ and
\begin{equation*}
    A_\gamma
    = \La r>0 ~:~ \forall x,y\in \gamma,~ \|x-y\|<2r \Rightarrow D_\gamma(x,y) \le 2r \arcsin\Lp\frac{\|x-y\|}{2r}\Rp\Ra~.
\end{equation*}
\Cref{lemma:metric_distortion_reach_characterization} states that $\tau_\gamma = \sup A_\gamma$ and standard reasoning over ordered sets shows that
\begin{equation*}
    \tau_{M,f}
    = \inf_\gamma \tau_\gamma
    = \inf_\gamma \sup A_\gamma
    \ge \sup \bigcap_\gamma A_\gamma
    = \sup A~.
\end{equation*}
Moreover, for all $x,y\in M$, $r\mapsto 2r \arcsin\bLp\|x-y\|/2r\bRp$ is a non-increasing function over $(\|x-y\|,\pinfty)$, which implies that for all $\gamma$, if $r \in A_\gamma$ then $s \in A_\gamma$ for all $0 < s\le r$.
In particular, it follows that the inequality above is in fact an equality, hence $\tau_{M,f} = \sup A$ which concludes the proof.
\end{proof}

Using \cref{prop:metric_distortion_reach_characterization_conformal}, we are able to lower bound the reach of any geodesic \ac{wrt} $D_{M,f}$.
\begin{proposition}
\label{prop:geodesic_reach}
Assume that $M$ has positive reach $\tau_M > 0$ and that $f$ is $\kappa$-Lipschitz and lower bounded by $\fmin >0$.
Then any conformal geodesic $\gamma$ \ac{wrt} $D_{M,f}$ has positive reach $\tau_\gamma>0$.
Precisely, the conformal reach of $M$ via $f$ is lower bounded as follows.
\begin{equation*}
    \tau_{M,f}
    \ge \cT_{M,f}
    \quad\textrm{where}\quad
    \cT_{M,f}
    \eqdef \min\Lp \frac{\tau_M}{2} ~,~ \frac{\fmin}{8\kappa} \Rp~.
\end{equation*}
\end{proposition}

The reasoning behind \cref{prop:geodesic_reach} is the following. If the reach of a geodesic $\gamma$ \ac{wrt} $D_{M,f}$ is small compared to $\tau_M$, then according to \cref{prop:metric_distortion_reach_characterization_conformal} there exists a path in $M$ significantly shorter than $\gamma$ Euclidean-wise. If $\tau_\gamma$ is also small compared to $\fmin/\kappa$, then this path is shown to also have a shorter conformal length than $\gamma$ due to the properties of $f$, which implies a contradiction with the geodesic nature of $\gamma$. This reasoning is detailed in \cref{app:geodesic_reach}.
Now, consider an arc-length parameterized geodesic $\gamma:I\to M$ \ac{wrt} $D_{M,f}$, \ie, such that $\|\dot\gamma\|=1$ almost everywhere over $I$. Being a geodesic, $\gamma$ has no self-intersection.
Then, stating that $\gamma$ has positive reach is equivalent to stating that $\gamma$ is a $\cC^{1,1}$ curve, \ie, that $\dot\gamma$ is Lipschitz \ac{wrt} the angular distance. Precisely, for all $t,s\in I$, 
\begin{equation}
    \label{eq:curve_reach_C11_angle}
    \angle(\dot\gamma(t),\dot\gamma(s))
    \le \frac{|t-s|}{\tau_\gamma}~.
\end{equation}
See \cite[Remark 4.20]{federerCurvatureMeasures1959} and \cite[Theorem 4]{lieutierManifoldsPositiveReach2024} for references.
\Cref{eq:curve_reach_C11_angle} allows to obtain a finer control on the approximation error of the polygonal paths on the graph by upper bounding efficiently the difference between small successive steps $\gamma(t)-\gamma(t-\delta)$ and $\gamma(t+\delta)-\gamma(t)$ of a geodesic path.

\begin{lemma}
\label{lemma:path_direction_lipschitz}
Let $\gamma : [0, |\gamma|] \to \RR^\damb$ be an arc-length parameterized curve without self-intersection and with positive reach $\tau_\gamma$.
Then for all $t_0\in[0, |\gamma|]$ and $\delta \in (0, \frac\pi2 \tau_\gamma]$ such that $[t_0-\delta, t_0+\delta] \subset [0,|\gamma|]$, the angle between the velocity vector of $\gamma$ at $t_0$ and the direction from $\gamma(t_0)$ to $\gamma(t_0+\delta)$ is upper bounded as follows:
\begin{equation}
    \label{eq:path_direction_lipschitz_1}
    \angle\bLp \gamma(t_0+\delta)-\gamma(t_0), \dot\gamma(t_0)\bRp
    \le \frac{\delta}{2\tau_\gamma}~.
\end{equation}
Moreover, denoting $v^- = \gamma(t_0)-\gamma(t_0-\delta)$ and $v^+ = \gamma(t_0+\delta)-\gamma(t_0)$, the difference between small steps of the path on both side from $t_0$ is upper bounded as follows:
\begin{equation}
    \label{eq:path_direction_lipschitz_2}
    \Lvv \frac{v^+}{\|v^+\|} - \frac{v^-}{\|v^-\|} \Rvv
    \le \frac{1}{\tau_\gamma} \min\bLp\|v^+\|, \|v^-\|\bRp~.
\end{equation}
\end{lemma}
\Cref{lemma:path_direction_lipschitz} is key to obtain an approximation error of the conformal metric proportional to the Hausdorff distance between the domain and a point cloud.
A proof is provided in \cref{app:proof:path_direction_lipschitz} and \cref{eq:path_direction_lipschitz_1} is illustrated by \cref{fig:path_direction_lipschitz}.
\begin{figure}[tb]
    \centering
    \includegraphics[width=0.5\textwidth]{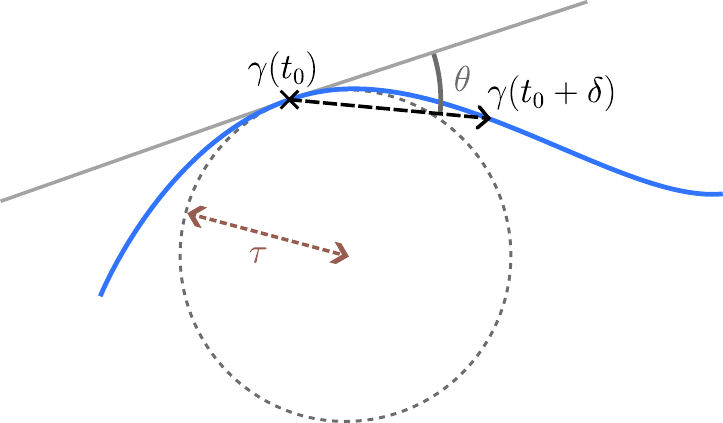}
    \caption{Bounding the angular velocity of a path with positive reach.}
    \label{fig:path_direction_lipschitz}
\end{figure}

\section{Polygonal Approximation of the Conformal Metric}
\label{sec:polygonal_approximation}

Consider a point cloud $X\subset M$ of points sampled from the domain. If the point cloud approximates well the domain, it is possible to estimate the conformal distance over $M$ through weighted polygonal paths built over $X$ within a small margin of error. Assuming that $f$ is known, the edges of the polygonal path are weighted using $f$ to approximate the conformal distance between their endpoints.
The degree of approximation of $M$ by $X$ is measured using the Hausdorff distance
\begin{equation*}
    \hausdorff(M,X)
    \eqdef \max\Lp \sup_{x\in M} d(x,X) ~,~ \sup_{x\in X} d(x,M) \Rp
\end{equation*}
which simplifies to $\hausdorff(M,X) = \sup_{x\in M} d(x,X)$ assuming that $X\subset M$. By definition, any point in $M$ is at distance at most $\hausdorff(M,X)$ from a point in $X$.

The idea behind the polygonal approximation is the following. Consider a geodesic ${\gamma \in \Gamma^\star_{M,f}(x,y)}$ between two points $x,y\in M$. If $\hausdorff(M,X)$ is small, there exists a polygonal path over $X$ that follows closely the trajectory of $\gamma$. Assuming that the edges of this path are equipped with appropriate weights to simulate the conformal distance between their endpoints, the total weight of the polygonal path should be close to the conformal length~$|\gamma|_f$. Moreover, since $\gamma$ is a geodesic, polygonal paths should not be able to have a weight much smaller than $|\gamma|_f$. This is ensured by using only short edges, which allows the weights to approximate well the conformal length by preventing shortcuts outside the domain.
In the end, the shortest weighted path on the graph is expected to retrieve $|\gamma|_f = D_{M,f}(x,y)$.

\subsection{Weighted Graphs}
\label{sec:weighted_graphs}

Let us now describe formally this construction.
The polygonal approximation of $D_{M,f}$ is defined as the metric of an appropriate weighted graph on a point cloud $X$.
\begin{definition}
\label{def:graphs}
Let $X \subset \RR^\damb$ be a finite point cloud.
\begin{itemize}
    \item For all $r>0$, the \emph{$r$-ball graph} over $X$ is denoted $G_r(X)$ and is the graph of vertex set $X$ and with an edge between points $x$ and $y$ if and only if $\|x-y\| \le r$.
    \item For all integer $k\ge1$, the \emph{$k$-nearest-neighbors graph} (or $k$-NN graph) over $X$ is denoted~$\cG_k(X)$ and is the graph of vertex set $X$ with an edge between two points $x$ and $y$ if and only if $x$ (or $y$) is among the $k$ nearest points of $X$ to $y$ (or $x$) excluding self.
\end{itemize}
In both cases, the parameter $r$ or $k$ is referred to as the \emph{threshold} of the graph.
\end{definition}

Ball graphs are more convenient to study whereas NN graphs are more practical, see \cref{sec:convergence_NN_graph}.
We define the polygonal metric in both contexts.
The graph used is equipped with a weight function to approximate the conformal length of its edges.
\begin{definition}
\label{def:weight_function}
Given a function $f:\RR^\damb\to\RR_+^*$, consider the weight functions
\[
    w_{f,q}(x,y)
    \eqdef \frac{\|x-y\|}{2(q-1)} \Lp f(x) + 2\sum_{k=2}^{q-1} f\Lp\frac{q-k}{q-1}x + \frac{k-1}{q-1}y\Rp + f(y) \Rp
\]
defined for any $q\ge2$ and $x,y \in \RR^\damb$, along with
\begin{equation*}
    w_{f,\infty}(x,y)
    \eqdef \|x-y\| \int_0^1 f\bLp(1-t)x+ty\bRp \d t~.
\end{equation*}
The parameter $q \in \{2,\dots,\infty\}$ is referred to as the \emph{resolution} of the weights.
\end{definition}

The weight function $w_{f,q}$ is meant to approximate the conformal distance between the endpoints of an edge. \Cref{prop:weight_distortion} below shows that it is indeed the case when the endpoints are close, which is the reason why graphs with small edges are used.
Recall that we have assumed without loss of generality that $f$ is defined over the whole space, allowing its evaluation outside $M$ when $q\ne2$. As $q$ grows, the weight $w_{f,q}$ becomes more accurate and converges to the weight with infinite resolution $w_{f,\infty}$, the latter being exactly the conformal length of the straight path. However, the computational cost is linear with $q$ and there is no closed-form formula for $w_{f,\infty}$ in general. On the other hand,
\begin{equation*}
    w_{f,2}(x,y)
    = \|x-y\| \frac{f(x)+f(y)}{2}
\end{equation*}
is the simplest choice for the weights and may be more suitable if evaluating $f$ is possible only at points in $X$.
Note that the weight $w_{f,q}$ is the optimal approximation of $w_{f,\infty}$ using only $q$ samples of the Lipschitz function $f$.
Under stronger assumptions on $f$ such as $\cC^k$ regularity, other definitions would be better suited.
\begin{proposition}
\label{prop:weight_distortion}
For all $x,y\in M$ such that $\|x-y\| \le \tau_{M,f}$,
\begin{equation}
    \label{eq:delta_q}
    \metricratioOld{w_{f,q}(x,y)}{D_{M,f}(x,y)}
    \le  \frac{\kappa}{4\fmin}\frac{\|x-y\|}{q-1} + \frac{\|x-y\|^2}{16\cT_{M,f}^2}~.
\end{equation}
and we denote $\delta_q(\|x-y\|)$ this upper bound. 
The first term in \cref{eq:delta_q} is to be interpreted as $0$ if $q=\infty$ and represents the offset between $w_{f,q}$ and $w_{f,\infty}$, whereas the second term represents the offset between $w_{f,\infty}$ and~$D_{M,f}$.
\end{proposition}

\Cref{prop:weight_distortion} shows that the conformal metric $D_{M,f}$ may be approximated locally by the weights $w_{f,q}$.
Note that for fixed resolution $q$ in the $r$-ball graph, the distortion between the weights and the conformal distances is linear in $r$. However, by choosing $q$ to be inversely proportional to $r$, the distortion becomes quadratic in $r$.
The proof for \cref{prop:weight_distortion} is provided in \cref{app:weight_props}.
We now introduce the polygonal approximation of the conformal metric as the metric of the weighted graph built on $X$.
\begin{definition}
\label{def:estimator_conform}
Consider a point cloud $X \subset \RR^\damb$, function $f : \RR^\damb\to\RR_+^*$ and parameters $r>0$ or $k\ge1$ and $q\in\{2,\dots,\infty\}$.
The polygonal metric associated with these parameters is defined between two points $x,y \in \RR^\damb$ as
\begin{equation}
    \label{eq:estimator_conform}
    \widehat D_{X,f}(x,y)
    \eqdef \min_{(x_0,\dots,x_K)} \sum_{k=0}^{K-1} w_{f,q}(x_k, x_{k+1})
\end{equation}
where the minimum is taken over the set of paths $(x_0,\dots,x_K)$ such that $x_0=x$ and $x_k=y$ in the graph $G$ that is chosen either as the $r$-ball graph $G_r(X\cup\{x,y\})$ or the $k$-NN graph $\cG_k(X\cup\{x,y\})$.
\end{definition}
The choice of a ball graph or a NN graph along with the threshold $r$ or $k$ and the resolution~$q$ of the weights are left implicit when writing $\widehat D_{X,f}$ to avoid heavy notations.
Note that when the endpoints $x$ and $y$ do not belong to the point cloud, the distance $\widehat D_{X,f}(x,y)$ is computed by adding them to the graph.
As a result, $\widehat D_{X,f}$ induces a distance over $X$ but not over~$M$. Indeed, when considering endpoints outside $X$, $\widehat D_{X,f}$ may not satisfy the triangular inequality due to the set of possible paths depending on the endpoints.

\subsection{Metric Approximation}
\label{sec:metric_approx}

To evaluate how efficient is the approximation of a metric, the error between the true distances and their estimation may be measured using either of the multiplicative loss functions
\begin{equation*}
    \ell_{\infty,M}\bLp D', D\bRp
    \eqdef \sup_{x\ne y\in M} \metricratio{D'(x,y)}{D(x,y)}
    \eqand
    l_{\infty,M}\bLp D'| D\bRp
    \eqdef \sup_{x\ne y\in M} \Lv 1-\frac{D'(x,y)}{D(x,y)} \Rv~.
\end{equation*}
In the case where $D'$ takes infinite values---which happens for instance when $D'$ is the metric of a non-connected graph---we let $\ell_{\infty,M}(D', D) = 1$ and $l_{\infty,M}(D'| D) = \pinfty$.
The inequality
\begin{equation}
    \label{eq:loss_comparison}
    \ell_{\infty,M}\bLp D', D\bRp
    \le l_{\infty,M}\bLp D'| D\bRp
    \le \frac{\ell_{\infty,M}\bLp D', D\bRp}{1-\ell_{\infty,M}\bLp D', D\bRp}
\end{equation}
holds in general, so that both losses are equivalent.
The loss $\ell_{\infty,M}$ is however easier to manipulate in some situations as it is symmetric and upper bounded by $1$.
Thanks to the regularity of geodesics stated by \cref{prop:geodesic_reach} and the local approximation of the conformal metric by the weights $w_{f,q}$ stated by \cref{prop:weight_distortion}, we are able to show that the conformal metric $D_{M,f}$ is approximated by the polygonal metric from \cref{def:estimator_conform} using the $r$-ball graph.
\begin{theorem}
\label{thm:geodesic_approx}
Let $X \subset M$ be a point cloud, $r>0$ and $q \in \{2,\dots,\infty\}$ two parameters.
Assume that $4\hausdorff(M,X) \le r \le \cT_{M,f}$.
Then the approximation $\widehat D_{X,f}$ defined in \cref{def:estimator_conform} using the $r$-ball graph with resolution $q$ satisfies
\begin{equation}
    \label{eq:thm:geodesic_approx}
    l_{\infty,M}\bLp\widehat D_{X,f}| D_{M,f}\bRp
    \le \frac{1}{32\cT_{M,f}}\frac{r}{q-1} + \frac{r^2}{8\cT_{M,f}^2} + 56\frac{\hausdorff(M,X)^2}{r^2}~.
\end{equation}
The first term in \cref{eq:thm:geodesic_approx} is to be interpreted as $0$ if $q=\infty$.
\end{theorem}
Let us describe the reasoning behind \cref{thm:geodesic_approx}. Given $x\ne y\in X$, paths in $G_r(X)$ cannot create significant shortcuts outside $M$ as $r \le \tau_{M,f}$. This is illustrated by \cref{prop:weight_distortion} from which it can be deduced that
\begin{equation}
    \label{eq:thm_geodesic_approach_case_1}
    \widehat D_{X,f}(x,y)
    \ge \bLp1-\delta_q(r)\bRp D_{M,f}(x,y) 
\end{equation}
and, assuming that $\|x-y\| \le r$,
\begin{equation}
    \label{eq:thm_geodesic_approach_case_2}
    \widehat D_{X,f}(x,y)
    \le \bLp1+\delta_q(r)\bRp D_{M,f}(x,y)~.
\end{equation}
As for when $\|x-y\| > r$, consider a geodesic $\gamma \in \Gamma^\star_{M,f}(x,y)$. Decompose $\gamma$ into sections of length at most $r-2\hausdorff(M,X)$ and for each intermediate point select a point in $X$ at distance at most $\hausdorff(M,X)$. This process draws a polygonal path in $G_r(X\cup\{x,y\})$ as it uses edges of length at most $r$. Summing the approximation error from \cref{prop:weight_distortion} between each edge of this path and the corresponding section of $\gamma$ eventually yields the upper bound
\begin{equation}
    \label{eq:thm_geodesic_approach_case_3}
    \widehat D_{X,f}(x,y)
    \le \Lp 1 + \delta_q(r) + c_1\frac{\hausdorff(M,X)}{\cT_{M,f}} + c_2\frac{\hausdorff(M,X)^2}{r^2} \Rp D_{M,f}(x,y) 
\end{equation}
when $\|x-y\| > r$ and where $c_1$ and $c_2$ are universal constants.
The positive reach of $\gamma$ and \cref{lemma:path_direction_lipschitz} play a crucial role in obtaining terms of order $\hausdorff(M,X)/\cT_{M,f}$ and $\hausdorff(M,X)^2/r^2$ instead of $\hausdorff(M,X)/r$.
Together, \cref{eq:thm_geodesic_approach_case_1,eq:thm_geodesic_approach_case_2,eq:thm_geodesic_approach_case_3} eventually imply \cref{eq:thm:geodesic_approx}.
The details of this reasoning are provided in \cref{app:proof:geodesic_approx}.
The upper bound in \cref{eq:thm:geodesic_approx} is made proportional to $\hausdorff(M,X)$ by setting
\begin{equation*}
    r \asymp \sqrt{\cT_{M,f}~\hausdorff(M,X)}
    \eqand
    q \asymp \sqrt{\frac{\cT_{M,f}}{\hausdorff(M,X)}}~.
\end{equation*}

In the case of the induced metric with $f=1$, the first term in the right-hand side of \cref{eq:thm:geodesic_approx} disappears, and the error is of order $r^2/\tau_M^2 + \hausdorff(M,X)^2/r^2$ similarly to the one obtained in \cite{arias-castroPerturbationBoundsProcrustes2020}.
In the general conformal case, we keep this magnitude of error by setting the resolution $q$ to be at least of order $1/r$.
If it is not possible to evaluate $f$ outside the point cloud $X$, setting $q=2$ instead yields an error of order $r/\cT_{M,f} + \hausdorff(M,X)^2/r^2$ which is slightly worse. A similar upper bound was achieved in \cite{aaronConvergenceRatesEstimators2018} for the induced metric and the term linear in $r$ was due to the geodesically smooth assumption used in the paper being slightly weaker than a positive reach assumption. In particular, the latter allows for an efficient local estimation in the form of \cref{eq:thm_geodesic_approach_case_2}.
The main takeaway of \cref{thm:geodesic_approx} is that the crucial assumption to achieve an approximation error proportional to $\hausdorff(M,X)$ is the positive reach. In particular, the domain need not be~$\cC^2$, or even a submanifold.

\subsection{Unknown Conformal Factor}
\label{sec:unknown_f}
In general, $f$ may not be known and needs to be estimated from the data.
The polygonal metric~$\widehat D_{X,g}$ from \cref{def:estimator_conform} can be defined for any function $g:\RR^\damb\to\RR_+^*$ without the need of Lipschitz and lower bound assumptions.
Then, the approximation error between~$\widehat D_{X,g}$ and~$D_{M,f}$ is upper bounded by the sum of the error \ac{wrt} the conformal factor and the error \ac{wrt} the domain.
\begin{lemma}
\label{lemma:estimator_convergence_both_errors}
Let $f:\RR^\damb\to\RR_+$ a function lower bounded by $\fmin$, $X\subset M$ a point cloud and $g:\RR^\damb\to\RR_+^*$ such that $\|g-f\|_\infty \le \frac12\fmin$. Then
\begin{equation}
    \ell_{\infty,M}\bLp\widehat D_{X,g}, D_{M,f}\bRp
    \le \frac{2}{\fmin}\|g-f\|_\infty + 2\ell_{\infty,M}\bLp\widehat D_{X,f}, D_{M,f}\bRp~.
\end{equation}
This result holds regardless of the type of graph, threshold and resolution as long as they are shared between $\widehat D_{X,f}$ and $\widehat D_{X,g}$.
\end{lemma}

The proof of \cref{lemma:estimator_convergence_both_errors} consists in straightforward computations and is deferred to \cref{app:proof:estimator_convergence_both_errors}.
As a consequence, in the context of estimation of the domain $M$ by a point cloud $X_n$ as in \cref{sec:random_samples} and of the conformal factor $f$ by a function $f_n : \RR^\damb \to \RR_+$, the rate of convergence of $\widehat D_{X_n,f_n}$ towards $D_{M,f}$ is the slowest rate of convergence among that of $\widehat D_{X_n,f}$ towards $D_{M,f}$ and that of $f_n$ towards $f$.
Recall that $f_n$ is not required to satisfy the Lipschitz and lower bounded assumptions for the polygonal metric $\widehat D_{X_n,f_n}$ to be defined and for \cref{lemma:estimator_convergence_both_errors} to hold. It may also be defined only over $X_n$ if the resolution is set to $q=2$.

\section{Estimation from Random Samples}
\label{sec:random_samples}

In this section we transpose \cref{thm:geodesic_approx} to a probabilistic setup where the point cloud $X$ is replaced with a random point cloud $X_n$ consisting of $n$ \iid samples from a probability measure~$\mu$ with support $M$.
The following assumptions on $\mu$ are needed to ensure that $X_n$ covers~$M$ efficiently as $n$ grows, allowing to deduce explicit convergence rates for the estimator $\widehat D_{X_n,f}$ from \cref{thm:geodesic_approx}.
\begin{definition}
\label{def:ahlfors}
Consider a probability measure $\mu$ with support $M\subset\RR^\damb$ and $d\ge2$.
\begin{itemize}
    \item The measure $\mu$ is $d$-standard with lower constant $c_\mu>0$ if for all $x\in M$ and $r>0$,
    \begin{equation*}
        \mu\bLp\oball(x,r)\bRp
        \ge c_\mu r^d \wedge 1
    \end{equation*}
    where $\oball(x,r)$ denotes the open ball centered at $x$ with radius $r$.
    \item The measure $\mu$ is $d$-Ahlfors with lower and upper constants $c_\mu>0$ and $C_\mu>0$ if for all $x\in M$ and $r>0$,
    \begin{equation*}
        c_\mu r^d \wedge 1
        \le \mu\bLp\oball(x,r)\bRp
        \le C_\mu r^d~.
    \end{equation*}
\end{itemize}
\end{definition}

If $\mu$ is $d$-standard with lower constant $c_\mu$, we also denote $L_\mu = c_\mu^{-1/d}$ which is related to the size of the support. The definition ensures that $\mu\bLp\oball(x,L_\mu)\bRp = 1$ for all $x\in M$, hence $L_\mu  \ge \diam(M)$. 
Moreover, $d$ acts as the intrinsic dimension of the support $M$ of~$\mu$. For instance, any measure with a density $\rho$ \ac{wrt} the volume measure of a submanifold of $\RR^\damb$ of dimension $d$ is $d$-Ahlfors if $\rho$ is bounded above and below.
The $d$-standard assumption ensures that the random point cloud $X_n$ converges to $M$ in Hausdorff distance. Ahlfors regularity is a stronger assumption and is needed to show the equivalence between ball graphs and NN graphs in \cref{prop:graph_ball_to_NN}.

\subsection{Convergence of the Ball Graph Estimator}
\label{sec:convergence_ball_graph}

We first discuss the case of a ball graph estimator.
It is known that if $\mu$ is $d$-standard, then $\hausdorff(M,X_n)$ is of order $L_\mu (\log(n)/n)^{1/d}$ at most, see \cref{app:proof:hausdorff_expectation}.
This convergence combined with \cref{thm:geodesic_approx} allows to derive convergence rates for the estimation of $D_{M,f}$.

\begin{theorem}
\label{thm:convergence_ball_graph}
Assume that $X_n$ is the result of $n$ \iid samples from a $d$-standard probability measure~$\mu$ with support $M$.
Consider the estimator $\widehat D_{X_n,f}$ from \cref{def:estimator_conform} using the $r$-ball graph with resolution $q$ where $r$ and $q$ are specified below.
Then there exists a constant $n_0$ depending on $L_\mu$, $\cT_{M,f}$ and $d$ such that for all $n\ge n_0$ the following holds.
\begin{itemize}
    \item If $r = 8\sqrt{L_\mu\cT_{M,f}} \Lp\frac{\log(n)}{n}\Rp^{\frac{1}{2d}}$ and $q\ge 1 + 4\frac{\cT_{M,f}}{r}$, then
        \begin{equation}
            \label{eq:thm:convergence_ball_graph_1}
            \EE\Lb \ell_{\infty,M}\bLp\widehat D_{X_n,f}, D_{M,f}\bRp \Rb
            \le 16\frac{L_\mu}{\cT_{M,f}} \Lp\frac{\log(n)}{n}\Rp^{\frac1d}~.
        \end{equation}
    \item If $r = 8L_\mu^{2/3}\cT_{M,f}^{1/3} \Lp\frac{\log(n)}{n}\Rp^{\frac{2}{3d}}$ and $q=2$, then
        \begin{equation}
            \label{eq:thm:convergence_ball_graph_2}
            \EE\Lb \ell_{\infty,M}\bLp\widehat D_{X_n,f}, D_{M,f}\bRp \Rb
            \le 8\Lp\frac{L_\mu}{\cT_{M,f}}\Rp^{\frac23} \Lp\frac{\log(n)}{n}\Rp^{\frac{2}{3d}}~.
        \end{equation}  
\end{itemize}
\end{theorem}

The first case in \cref{thm:convergence_ball_graph} includes $q=\infty$. There is no use setting $q$ to be greater than the indicated threshold of order $\cT_{M,f} / r$, as the term in the upper bound depending on $q$ becomes negligible past this threshold.
Note that the optimal choice for $r$ and~$q$ requires the knowledge of $L_\mu$, $\cT_{M,f}$ and most importantly $d$. \cref{thm:convergence_ball_graph} is stated with these choices of parameters to highlight the optimal theoretical dependency in these constants. 
For practical purposes, the knowledge of $L_\mu$ and $\cT_{M,f}$ is not actually needed as, in general, any choice of $r \asymp (\log(n)/n)^{1/2d}$ and $q \gtrsim 1/r$ yields a convergence rate of $(\log(n)/n)^{1/d}$.
Likewise, any choice of $r \asymp (\log(n)/n)^{2/3d}$ and $q = 2$ yields a convergence rate of $(\log(n)/n)^{2/3d}$.
It is not possible however to get rid of the dependency in $d$ for the range parameter $r$ without altering the convergence rate.
To circumvent this issue, one may use the NN graph instead, for which the optimal choice of the parameter is shown in \cref{sec:convergence_NN_graph} to be $k = \sqrt{n}$ which does not depend on $d$.
\Cref{thm:convergence_ball_graph} is derived directly from \cref{thm:geodesic_approx} and the convergence of $X_n$ to $M$ in Hausdorff distance. The precise computations are deferred to \cref{app:proof::convergence_ball_graph}.
In the case of the induced metric, the weight function $w_{f,q}$ becomes the Euclidean distance regardless of the resolution and $\cT_{M,f}$ is replaced with $\tau_M$.
\begin{corollary}
\label{corol:induced_metric_approx}
Under the same assumptions as in \cref{thm:convergence_ball_graph} and when $f=1$, there exists a constant $n_0$ depending on $L_\mu$, $\tau_M$ and $d$ such that for all $n\ge n_0$, setting ${r = 8\sqrt{L_\mu\tau_M} (\log(n)/n)^{1/2d}}$, the estimator $\widehat D_{X_n,1}$ satisfies
\begin{equation*}
    \EE\Lb \ell_{\infty,M}\bLp\widehat D_{X_n,1}, D_M\bRp \Rb
    \le 16\frac{L_\mu}{\tau_M} \Lp\frac{\log(n)}{n}\Rp^{\frac{1}{d}}~.
\end{equation*}
\end{corollary}
Recall that in the case of a submanifold of class $\cC^k$ with $k\ge2$ and dimension $d$, the optimal convergence rate for the estimation of the induced metric is $n^{-k/d}$ \cite{aamariOptimalReachEstimation2023}.
\Cref{corol:induced_metric_approx} extends the upper bound to any set of positive reach with the convergence rate $n^{-1/d}$.

\subsection{Convergence of the Nearest-Neighbors Graph Estimator}
\label{sec:convergence_NN_graph}

We now treat the case of a NN graph estimator, by observing that for adequate choice of $r$ and $k$, the $r$-ball graph and $k$-NN graph are similar.
Indeed, assuming that $\mu$ is $d$-Ahlfors regular, a ball of radius $r$ is expected to contain between $c_\mu n r^d$ and $C_\mu n r^d$ points from $X_n$.
\begin{proposition}
\label{prop:graph_ball_to_NN}
Let $\mu$ be a $d$-Ahlfors measure as defined in \cref{def:ahlfors} and $X_n$ a point cloud sampled \iid from $\mu$. Let $k\ge1$, $\epsilon\in(0,1)$ and define
\begin{equation*}
    r_- = \Lp(1-\epsilon)\frac{k}{C_\mu(n-1)}\Rp^{\frac1d}
    \eqand
    r_+ = \Lp\frac{1}{1-\epsilon}\frac{k}{c_\mu(n-1)}\Rp^{\frac1d}~.
\end{equation*}
Then with probability at least $1 - 2n e^{-\epsilon^2k / 2}$, the $k$-NN graph $\cG_k(X_n)$ is enclosed between two ball graphs $G_{r_-}(X_n)$ and $G_{r_+}(X_n)$, that is
\begin{equation*}
    \PP\Lp G_{r_-}(X_n)
    \subset \cG_k(X_n)
    \subset G_{r_+}(X_n) \Rp
    \ge 1 - 2n \exp\Lp-\frac{\epsilon^2}{2}k\Rp
\end{equation*}
where the inclusion $\subset$ refers to the inclusion of edge sets.
\end{proposition}

\Cref{prop:graph_ball_to_NN} shows that if $r$ and $k$ are chosen such that $k \asymp nr^d$, the $k$-NN graph and the $r$-ball graph are very similar with high probability.
This result stems from a common intuition and a detailed proof is provided in \cref{app:graph_ball_to_NN} for completeness.
\Cref{thm:convergence_ball_graph} is then extended to $k$-NN graphs by choosing $k \asymp \sqrt{n\log(n)}$, which does not depend on $d$.
\begin{theorem}
\label{thm:convergence_NN_graph}
Assume that $X_n$ is the result of $n$ \iid samples from a $d$-Ahlfors probability measure~$\mu$ with support $M$.
Consider the estimator $\widehat D_{X_n,f}$ from \cref{def:estimator_conform} using the $k$-NN graph with parameters $k=\big\lceil\sqrt{n\log(n)}~\big\rceil$ and $q=\lceil n^{1/4}\rceil$.
Then there exists a constant $n_0$ depending on $L_\mu$, $\cT_{M,f}$ and $d$ such that for all $n\ge n_0$ the following holds.
\begin{equation*}
    \EE\Lb \ell_{\infty,X_n}\bLp\widehat D_{X_n,f}, D_{M,f}\bRp \Rb
    \le C \Lp\frac{\log(n)}{n}\Rp^{\frac{1}{d}}
\end{equation*}
where $C$ is a constant depending on $L_\mu$ and $\cT_{M,f}$.
\end{theorem}

Notice that in \cref{thm:convergence_NN_graph} the loss is over $X_n$ instead of $M$ like in \cref{thm:convergence_ball_graph}. The  same result could be stated over $M$, although the proof would be more tedious and is therefore omitted here.
Recall that no knowledge on either $d$, $c_\mu$, $C_\mu$, $\tau_M$, $\kappa$ or $\fmin$ is necessary to compute the estimator, so that it can be used in practice.
Setting the resolution $q$ to be $(n/\log(n))^{1/2d}$ would be sufficient to achieve the same upper bound, however this choice requires the knowledge of $d$.
On the other hand, when setting the resolution to $q=2$ the optimal choice of $k$ can be shown similarly to be $k = n^{1/3}\log(n)^{2/3}$ and to yield a convergence rate of order $(\log(n)/n)^{2/3d}$.
\Cref{thm:convergence_NN_graph} is a direct consequence of \cref{thm:convergence_ball_graph,prop:graph_ball_to_NN} and the precise computations are provided in \cref{app:proof:convergence_NN_graph}.
Regarding the case of the induced metric, the same statement as in \cref{corol:induced_metric_approx} holds for NN graphs. Namely, the estimator $\widehat D_{X_n,1}$ defined over the $k$-NN graph with $k=\big\lceil\sqrt{n\log(n)}~\big\rceil$ converges to $D_M$ at rate $(\log(n)/n)^{1/d}$.

Let us now discuss the algorithmic complexity of the $k$-NN estimator with resolution $q$.
In general, building the $k$-NN graph over a point cloud $X$ of size $n$ is done in $\cO(n^2\damb)$ time when the ambient dimension $\damb$ is large.
Now, consider two points $x,y\in\RR^\damb$ and assume that $f$ can be evaluated at any point with cost $c_f$. Computing the edges of the graph $\cG_k(X\cup\{x,y\})$ takes $\cO(nkqc_f)$ time as the amount of edges in the graph is $\cO(nk)$ and the weight of each edge uses $q$ evaluations of $f$.
Finally, computing the infimum that defines $D_{X,f}(x,y)$ in \cref{def:estimator_conform} using Dijkstra's algorithm takes $\cO(n\log(n) + nk)$ time.
Overall, if $k=\big\lceil\sqrt{n\log(n)}~\big\rceil$ and $q=\lceil n^{1/4}\rceil$, the time complexity is $\cO(n^2\damb + n^{7/4}\log(n)^{1/2} c_f)$.

\subsection{Minimax lower bound}
\label{sec:minimax_lower_bound}

We now study the worst case performance of any estimator of the induced metric $D_M$.
\begin{theorem}
\label{thm:minimax_lower_bound}
Denote $\mathfrak D_n$ the set of all estimators $\widehat D$ of the induced metric based on $n$ samples, that given any point cloud $X$ of $n$ points in $\RR^\damb$ provides a function $\widehat D_X : \RR^\damb\times\RR^\damb \to \RR_+$.
Let $2\le d\le\damb$, $L,\tau > 0$ and $\cM(d,L,\tau)$ be the set of all $d$-standard measures $\mu$ with $L_\mu\le L$ and support $M_\mu \subset \RR^\damb$ that has positive reach lower bounded by~$\tau$.
Then there exists two constants $C>0$ and $n_0$ depending on $d$, $L$ and $\tau$ such that for all $n \ge n_0$,
\begin{equation}
    \label{eq:thm:minimax_lower_bound}
    \inf_{\hat D \in \mathfrak D_n} \sup_{\mu\in\cM(d,L,\tau)} \EE_{X \sim \mu^{\otimes n}}\Lb \ell_{\infty,M_\mu}\bLp\widehat D_X, D_{M_\mu}\bRp \Rb
    \ge C \Lp\frac1n\Rp^{\frac{1}{d-1/2}}~.
\end{equation}
\end{theorem}

Notice that \cref{thm:minimax_lower_bound} only addresses the case of the induced metric. This is not a loss of generality and in fact highlights the fact that the conformal change does not make the problem any harder, as we have established in \cref{sec:regularity} that geodesics have the same regularity as in the case of the induced metric.
Given any other function $f$ satisfying the assumptions for a conformal metric, the same lower bound may be obtained with a similar reasoning as what follows, albeit with more technicalities.
The rate $n^{-1/(d-1/2)}$ is faster than the rate $n^{1/d}$ obtained in \cref{thm:convergence_ball_graph,thm:convergence_NN_graph}, although the difference becomes negligible when $d$ is large. The nature of the minimax convergence rate remains therefore open. Under stronger assumptions on the domain, this question is already solved---see \cite{aamariOptimalReachEstimation2023}---which we discuss in \cref{sec:k_d}.

\Cref{thm:minimax_lower_bound} is based on Le Cam's method \cite{yuAssouadFanoCam1997}, for which a statement adapted to our setup is given in \cref{lemma:lecam}.
The method consists in finding two measures $\mu_1$ and $\mu_2$ in $\cM(d,L,\tau)$ that are at most $1/n$ apart in total variation distance and such that the relative difference between $D_{M_1}$ and $D_{M_2}$ is of order at least $n^{-1/(d-1/2)}$.

Consider $\mu_1$ the uniform probability on the cube $M_1 = [-\alpha L,\alpha L]^d \times \{0\}^{\damb-d} \subset \RR^\damb$ where $\alpha>0$.
Let $0 < \epsilon \le \alpha L \wedge \tau$ and consider $M_2$ the result of removing from $M_1$ its intersection with a ball of radius $\tau$ centered at $(0,t,t,\dots,t)$ for some $t > \alpha L$ such that the ball intersects the edge from $(-\alpha L,\alpha L,\alpha L,\dots,\alpha L)$ to $(\alpha L,\alpha L,\alpha L,\dots,\alpha L)$ at two points $x$ and $y$ that are $2\epsilon$ apart, as pictured in \cref{fig:carved_cube}.
Denote $\mu_2$ the uniform probability over $M_2$, which has reach exactly $\tau$ due to the carved area. By choosing $\alpha$ small enough, which depends only on $d$, $\mu_1$ and $\mu_2$ are $d$-standard with lower constant $L^{-d}$.
Then $\mu_1$ and $\mu_2$ both belong to $\cM(d,L,\tau)$.
The volume of the carved area is of order $\epsilon(\epsilon^2)^{d-1}$ as it spans a length $2\epsilon$ between $x$ and~$y$ and a length of order $\epsilon^2$ for every other dimension of the cube. This implies that the total variation distance between $\mu_1$ and $\mu_2$ is of order $\epsilon^{2d-1}$.
Moreover, The distance from $x$ to~$y$ goes from $2\epsilon$ in $M_1$ to $2\tau\arcsin(\epsilon/\tau)$ in $M_2$, following the red arc of radius $\tau$ in \cref{fig:carved_cube}. This implies that the relative difference between both distances is of order $\epsilon^2$.
Choosing $\epsilon \asymp n^{-1/2d-1}$ so that the total variation distance is at most $1/n$ eventually yields the desired bound.
This reasoning is detailed in \cref{app:minimax}.

Using this method, the above construction leads to a minimax lower bound that does not quite fit the upper bound from \cref{thm:convergence_ball_graph}.
If a better lower bound can be obtained, it is likely that it would require a different technique.
Indeed, given fixed endpoints $x$ and $y$ at distance $\epsilon$ from each other and belonging to the intersection of two domains $M_1$ and $M_2$ with lower bounded reach, the relative difference between $D_{M_1}(x,y)$ and $D_{M_2}(x,y)$ is of order at most $\epsilon^2$ due to the reach assumption. Then, in order to achieve such distortion, it is reasonable to expect that the volume of the difference between $M_1$ and $M_2$ needs to be of order at least $\epsilon$ in the direction of $x-y$ and at least $\epsilon^2$ in the other directions to ensure that the uniform measures are $d$-standard, which leads to a volume of order $\epsilon^{2d-1}$ as in the previous construction.

\begin{figure}[tb]
\centering
\tdplotsetmaincoords{60}{150} 
\begin{tikzpicture}[tdplot_main_coords,scale=3]
    \def\eps{0.6} 

    \def\r{0.5}   
    \pgfmathsetmacro{\delta}{sqrt(\r*\r - (\eps/2)*(\eps/2))} 
    \pgfmathsetmacro{\theta}{asin(\eps/(2*\r))}              

    \def\rface{0.4}   
    \pgfmathsetmacro{\deltaface}{sqrt(\rface*\rface - (\eps/2)*(\eps/2))}
    \pgfmathsetmacro{\thetaface}{asin(\eps/(2*\rface))}

    \def\N{80}

    \coordinate (O) at (0,0,0);
    \coordinate (X) at (1,0,0);
    \coordinate (Y) at (0,1,0);
    \coordinate (Z) at (0,0,1);
    \coordinate (XY) at (1,1,0);
    \coordinate (XZ) at (1,0,1);
    \coordinate (YZ) at (0,1,1);
    \coordinate (XYZ) at (1,1,1);
    \coordinate (x) at (1,1,{0.5-\eps/2});
    \coordinate (y) at (1,1,{0.5+\eps/2});

    \begin{scope}
    \clip (Z)--(XZ)--(XYZ)--(YZ)--cycle;
    \fill[pattern=north east lines,pattern color=gray!70] (Z)--(XZ)--(XYZ)--(YZ)--cycle;
    \end{scope}

    \foreach \i in {0,...,\N} {
    \pgfmathsetmacro{\ang}{180 + \thetaface - (\i)*(2*\thetaface)/\N} 
    \coordinate (arcA\i) at
        ({1}, {1 + \deltaface + \rface*cos(\ang)}, {0.5 + \rface*sin(\ang)});
    }
    \foreach \i in {0,...,\N} {
    \pgfmathsetmacro{\ang}{180 + \thetaface - (\i)*(2*\thetaface)/\N} 
    \coordinate (arcB\i) at
        ({1 + \deltaface + \rface*cos(\ang)}, {1}, {0.5 + \rface*sin(\ang)});
    }
    \foreach \i in {0,...,\N} {
    \pgfmathsetmacro{\ang}{180 + \theta - (\i)*(2*\theta)/\N} 
    \coordinate (arcC\i) at
        ({1 + \delta + \r*cos(\ang)}, {1 + \delta + \r*cos(\ang)}, {0.5 + \r*sin(\ang)});
    }

    \begin{scope}
    \clip (X) -- (XY) -- (x)
        \foreach \i in {0,...,\N} { -- (arcA\i) }
        -- (y) -- (XYZ) -- (XZ) -- cycle;
    \fill[pattern=north east lines,pattern color=gray!60] (X)--(XY)--(XYZ)--(XZ)--cycle;
    \end{scope}
    \begin{scope}
    \clip (Y) -- (XY) -- (x)
        \foreach \i in {0,...,\N} { -- (arcB\i) }
        -- (y) -- (XYZ) -- (YZ) -- cycle;
    \fill[pattern=north east lines,pattern color=gray!60] (Y)--(XY)--(XYZ)--(YZ)--cycle;
    \end{scope}

    \draw[thick] (X)--(XY)--(Y);
    \draw[thick] (Z)--(XZ)--(X);
    \draw[thick] (Y)--(YZ)--(XYZ);
    \draw[thick] (Z)--(YZ);
    \draw[thick] (XZ)--(XYZ);
    \draw[thick] (XY)--(x);
    \draw[thick] (y)--(XYZ);

    \draw[line width=1pt] (arcA0)
    \foreach \i in {1,...,\N} { -- (arcA\i) };
    \draw[line width=1pt] (arcB0)
    \foreach \i in {1,...,\N} { -- (arcB\i) };
    \draw[line width=1pt, red] (arcC0)
    \foreach \i in {1,...,\N} { -- (arcC\i) };

    \draw[line width=1.2pt] ($(x) + (-0.02,0.02,0)$) -- ($(x) + (0.02,-0.02,0)$);
    \node[right=4pt] at (x) {$x$};
    \draw[line width=1.2pt] ($(y) + (-0.02,0.02,0)$) -- ($(y) + (0.02,-0.02,0)$);
    \node[right=4pt] at (y) {$y$};
\end{tikzpicture}
\caption{Carving an edge of the cube in $\RR^3$.}
\label{fig:carved_cube}
\end{figure}
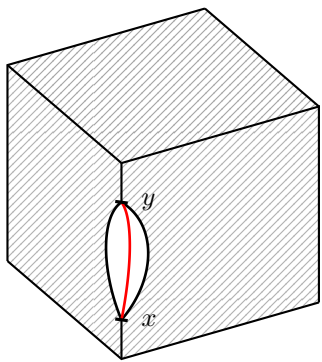

\subsection{Case of a Smooth Manifold}
\label{sec:k_d}

Recall that the minimax convergence rate for the induced metric of a $\cC^k$ manifold of dimension $d$ is $n^{-k/d}$ \cite{aamariOptimalReachEstimation2023}. The methods that achieve such rates are however not computationally feasible in practice as they rely on manifold reconstruction via non-discrete sets. For instance, the optimal convergence rate in the $\cC^2$ case was achieved by \cite{arias-castroPerturbationBoundsProcrustes2020} using the tangential Delaunay complex.
Precisely, these methods build an approximation of the manifold that is at Hausdorff distance of order $n^{-k/d}$, then state that the induced metric over this reconstruction is an estimation of the original metric with an error proportional to the Hausdorff distance, that is of order $n^{-k/d}$.

In order to get a concrete estimator, one may use such manifold reconstruction, then sample a fine $n^{-k/d}$-net over it which approximates the original domain with the same Hausdorff error. Our work then implies that the polygonal metric over this net would be an estimation of the original metric with minimax optimal error of order $n^{-k/d}$.
Such construction would however be very costly as the point cloud size would grow from $n$ to $n^k$.

\section{Practical Examples of Conformal Factors}
\label{sec:particular_f}

We discuss two examples of conformal factors associated with a measure from the literature.

\subsection{Density as a conformal factor}
If $\mu$ is a measure with density $\rho$ \ac{wrt} the volume measure over a submanifold $M$, the conformal change via $f = \rho^{-\beta}$ for some parameter $\beta>0$ is sometimes referred to as the \emph{Fermat distance} due to the parallel with the Fermat principle in optics---that may however be observed with any conformal change. This metric has been applied to various topological data analysis and learning problems, \eg, in \cite{fernandezIntrinsicPersistentHomology2023,garciatrillosFermatDistancesMetric2024}, as it tends to accentuate features and disparities in the data.
Moreover, the metric can be estimated in a simple fashion by using the same overall reasoning as ours but using weights of the form $w(x,y) = \|x-y\|^\alpha$ where $\alpha>1$ is a parameter depending on $\beta$ and $d$ \cite{groismanNonhomogeneousEuclideanFirstpassage2022,hwangShortestPathRandom2016}. However, this kind of estimation does not feature any known convergence rate.
Using our estimator instead provides an alternative method with quantitative guarantees, granted that an estimator of the density is available.
To this extent, density estimation is a well-studied problem with many propositions in the literature. For instance, \cite{berenfeldDensityEstimationUnknown2021} provides a kernel density estimator $\rho_n$ that converges in $L^p$ norm towards $\rho$ at a rate of $n^{-1/d+1}$ under our assumptions and provided that $M$ is a $\cC^1$ submanifold and that $\rho$ is also $\cC^1$.
Under these assumptions, setting $f_n = \rho_n^{-\beta}$ yields a convergence rate of $n^{-1/d}$ according to \cref{lemma:estimator_convergence_both_errors}, at the cost of a more complex computation than the usual discrete Fermat distance studied in \cite{groismanNonhomogeneousEuclideanFirstpassage2022}.

\subsection{Distance-to-measure as a conformal factor}
In general, given a parameter $m\in(0,1)$, the \emph{distance-to-measure} $d_{\mu,m} : \RR^\damb \to \RR_+$ introduced in \cite{chazalGeometricInferenceMeasures2011} is defined for any measure $\mu$ over~$\RR^\damb$. It is $1$-Lipschitz and lower bounded by a positive value as long as $\mu$ has no atom, hence satisfies the assumption for our work.
A slightly different setup where paths are allowed to leave the domain under a specific constraint is studied in \cite{taupinFermatDistancetoMeasureRobust2025}, where it is argued that this metric should behave similarly to the conformal metric associated with the density but with more stability \ac{wrt} the measure.
The distance-to-measure is shown \cite{chazalRatesConvergenceRobust2016} to be estimated from $n$ \iid samples of the underlying measure with convergence rate $n^{-1/2}$, which is faster than $n^{-1/d}$ hence does not impact the convergence speed of the estimator of the conformal metric according to \cref{lemma:estimator_convergence_both_errors}.

\section{Conclusion}
\label{sec:conclusion}

In this work, we have shown that under a reach assumption on the domain and Lipschitz lower bounded assumptions on the conformal change, the conformal metric has the same regularity as the induced metric in the sense of geodesics having positive reach.
We have also shown that positive reach is a sufficient assumption to ensure metric estimation from a polygonal metric with error proportional to the Hausdorff distance between the point cloud and the original domain.
This leads to a convergence rate of order $n^{-1/d}$ for the estimation of the conformal metric using $n$ \iid samples from a $d$-standard measure, which in particular applies to the induced metric of any set with positive reach.

\bibliography{biblio}

\appendix

\section{Intermediate Results}

This section is devoted to the proofs of some intermediate results used in the main text.

\subsection{Rate of Change of a Curve with Positive Reach}
\label{app:proof:path_direction_lipschitz}

In this section we prove \cref{lemma:path_direction_lipschitz}.
Let $v^+ = \gamma(t_0+\delta) - \gamma(t_0)$ and $v^- = \gamma(t_0) - \gamma(t_0-\delta)$.
The angle between two non-zero vectors is the arccosine of the dot product of the normalized vectors, thus
\begin{align}
    \label{eq:proof:path_direction_lipschitz_1}
    \angle\bLp v^+, \dot\gamma(t_0)\bRp
    = \arccos\Lp\Lan \frac{v^+}{\|v^+\|}, \dot\gamma(t_0) \Ran\Rp~.
\end{align}
Since $v^+ = \gamma(t_0+\delta) - \gamma(t_0) = \int_{t_0}^{t_0+\delta} \dot\gamma(t) \d t$,
\begin{align}
    \label{eq:proof:path_direction_lipschitz_2}
    \bLan v^+, \dot\gamma(t_0) \bRan
    = \int_{t_0}^{t_0+\delta} \bLan \dot\gamma(t), \dot\gamma(t_0) \bRan \d t
    = \int_{t_0}^{t_0+\delta} \cos\bLp\angle\bLp\dot\gamma(t) , \dot\gamma(t_0)\bRp\bRp \d t~.
\end{align}
Moreover, $\|v^+\| = \|\gamma(t_0+\delta) - \gamma(t_0)\| \le \delta$ as $\gamma$ is $1$-Lispchitz, and $\cos$ is non-increasing over $[0,\frac\pi2]$. Therefore, \cref{eq:proof:path_direction_lipschitz_2,eq:curve_reach_C11_angle} imply that
\begin{align}
    \label{eq:proof:path_direction_lipschitz_3}
    \Lan \frac{v^+}{\|v^+\|}, \dot\gamma(t_0) \Ran
    = \frac{1}{\|v^+\|} \int_{t_0}^{t_0+\delta} \cos\bLp\angle\bLp\dot\gamma(t) , \dot\gamma(t_0)\bRp\bRp \d t
    \ge \frac1\delta \int_{t_0}^{t_0+\delta} \cos\Lp\frac{t-t_0}{\tau_\gamma}\Rp \d t~.
\end{align}
Now, since $\arccos$ is non-increasing and concave over $[0,1]$, \cref{eq:proof:path_direction_lipschitz_1,eq:proof:path_direction_lipschitz_3} together with Jensen's inequality yield
\begin{align*}
    \angle\bLp v^+, \dot\gamma(t_0)\bRp
    \le \arccos\Lp \frac1\delta \int_{t_0}^{t_0+\delta} \cos\Lp\frac{t-t_0}{\tau_\gamma}\Rp \d t \Rp
    \le \frac1\delta \int_{t_0}^{t_0+\delta} \frac{t-t_0}{\tau_\gamma} \d t
    = \frac{\delta}{2\tau_\gamma}
\end{align*}
which concludes \cref{eq:path_direction_lipschitz_1}.
The same result holds when comparing $\dot\gamma(t_0)$ to a small variation before $t_0$ instead of after, therefore
\begin{align*}
    \angle\bLp v^+, v^- \bRp
    \le \angle\bLp v^+, \dot\gamma(t_0)\bRp + \angle\bLp v^-, \dot\gamma(t_0)\bRp
    \le \frac{\delta}{\tau_\gamma}~.
\end{align*}
Moreover, applying \cref{eq:arcsin_inequality} to the induced metric on the curve $\gamma$ between endpoints $\gamma(t_0)$ and either $\gamma(t_0+\delta)$ or $\gamma(t_0-\delta)$ shows that
\begin{align*}
    \delta \le \min\Lp 2\tau_\gamma\arcsin\Lp\frac{\|v^+\|}{2\tau_\gamma}\Rp, 2\tau_\gamma\arcsin\Lp\frac{\|v^-\|}{2\tau_\gamma}\Rp \Rp~.
\end{align*}
Finally,
\begin{align*}
    \Lvv \frac{v^+}{\|v^+\|} - \frac{v^-}{\|v^-\|} \Rvv
    = 2\sin\Lp\frac12\angle\bLp v^+, v^- \bRp\Rp
    \le 2\sin\Lp\frac{\delta}{2\tau_\gamma}\Rp
    \le \frac{1}{\tau_\gamma} \min\bLp\|v^+\|, \|v^-\|\bRp
\end{align*}
which concludes \cref{eq:path_direction_lipschitz_2}.

\subsection{Convergence of a Point Cloud in Hausdorff Distance}
\label{app:proof:hausdorff_expectation}

Recall the convergence in probability of the Hausdorff distance between an \iid point cloud and the support of a $d$-standard measure.
\begin{lemma}{\cite[Theorem 2]{chazalConvergenceRatesPersistence2015}}
\label{lemma:hausdorff_proba_bound}
Let $X_n$ be sampled \iid from a $d$-standard measure~$\mu$ with support $M$ and lower constant $c_\mu$. Then for all $\epsilon > 0$,
\begin{align*}
    \PP\Lp \hausdorff(M,X_n) > \epsilon \Rp
    \le \frac{4^d}{c_\mu\epsilon^d} \exp\Lp-n \frac{c_\mu\epsilon^d}{2^d}\Rp~.
\end{align*}
\end{lemma}

To establish our results in \cref{sec:random_samples}, we require the convergence of the Hausdorff distance in expectation. 
\begin{corollary}
\label{corol:hausdorff_expectation}
Under the same assumptions as in \cref{lemma:hausdorff_proba_bound} and assuming that $n\ge8$, for all $1\le p \le 2d$ it holds that
\begin{align*}
    \EE\bLb\hausdorff(M, X_n)^p\bRb
    \le \frac{2^{3p/2}}{c_\mu^{p/d}} \Lp\frac{\log(n)}{n}\Rp^{p/d}~.
\end{align*}
\end{corollary}

\begin{proof}
Denote $Z = \hausdorff(M, X_n)$, $c = 4^d / c_\mu$ and $\alpha = c_\mu n / 2^d$ for simplicity, so that
\begin{align}
    \label{eq:proof_hausdorff_expectation}
    \PP(Z > t)
    \le \frac{c}{t^d} \exp(-\alpha t^d)
\end{align}
for all $t>0$ according to \cref{lemma:hausdorff_proba_bound}. Then, the expectation of $Z^{2d}$ can be written as
\begin{align*}
    \EE\bLb Z^{2d}\bRb
    = \int_0^\pinfty \PP(Z^{2d} > u) \d u
    = \int_0^\pinfty 2d t^{2d-1}\PP(Z > t) \d t~.
\end{align*}
Given any $t_0>0$, upper bounding the probability by $1$ when $t\le t_0$ and using \cref{eq:proof_hausdorff_expectation} when $t>t_0$ instead yields
\begin{align*}
    \EE\bLb Z^{2d}\bRb
    \le \int_0^{t_0} 2d t^{2d-1} \d t + 2c \int_{t_0}^\pinfty d t^{d-1} \exp(-\alpha t^d) \d t
    = t_0^{2d} + \frac{2c}{\alpha} \exp(-\alpha t_0^d)~.
\end{align*}
Then, setting $t_0 = \bLp\log(n)/\alpha\bRp^{1/d}$ and replacing $c$ and $\alpha$ with their definition yields
\begin{align*}
    \EE\bLb Z^{2d}\bRb
    \le 2^{2d} \Lp\frac{\log(n)}{c_\mu~n}\Rp^2 + \frac{2^{3d+1}}{c_\mu^2~n^2}
    = \Lp 2^{2d} + \frac{2^{3d+1}}{\log(n)^2} \Rp \Lp\frac{\log(n)}{c_\mu~n}\Rp^2
    \le 2^{3d} \Lp\frac{\log(n)}{c_\mu~n}\Rp^2
\end{align*}
by lower bounding $\log(n)^2 \ge 4$ when $n\ge8$.
Finally, using Jensen's inequality yields the desired inequality for all $1\le p\le 2d$, that is
\begin{align*}
    \EE\bLb Z^p\bRb
    \le \EE\bLb Z^{2d}\bRb^{\frac{p}{2d}}
    &\le 2^{\frac{3p}{2}} \Lp\frac{\log(n)}{c_\mu~n}\Rp^{\frac{p}{d}}~.
\end{align*}
\end{proof}

\subsection{Ball Graphs and Nearest-Neighbors Graphs}
\label{app:graph_ball_to_NN}

In this section we prove \cref{prop:graph_ball_to_NN}. Denote $X_n = (x_1,\dots,x_n) \sim \mu^{\otimes n}$ the point cloud.
As a preliminary result, denote for all $n\in\NN^*$, $p\in(0,1)$ and $k\in\RR$, $\phi(n,p,k)$ the probability that a random variable following the binomial law with parameters $n$ and $p$ is greater than $k$. Denote $\psi(n,p,k)$ the probability that a random variable with same law is smaller than $k$. Then for all $k > np$,
\begin{align}
    \label{eq:chernoff1}
    \phi(n,p,k)
    \le \exp\Lp \Lp1 + \log\Lp\frac{np}{k}\Rp - \frac{np}{k}\Rp k\Rp~,
\end{align}
and for all $k < np$,
\begin{align}
    \label{eq:chernoff2}
    \psi(n,p,k)
    \le \exp\Lp \Lp1 + \log\Lp\frac{np}{k}\Rp - \frac{np}{k}\Rp k\Rp~.
\end{align}
Indeed, both quantities may be upper bounded using a Chernoff bound by
\begin{align*}
    \Lp\frac{np}{k}\Rp^k \Lp\frac{n-np}{n-k}\Rp^{n-k}
    = \Lp\frac{np}{k}\Rp^k \Lp1+\frac{k-np}{n-k}\Rp^{n-k}
    \le \Lp\frac{np}{k}\Rp^k e^{k-np}~.
\end{align*}
Now consider $k\in\{1,\dots,n-1\}$, $\epsilon\in(0,1)$ and let
\begin{align*}
    r_- = \Lp(1-\epsilon)\frac{k}{C_\mu(n-1)}\Rp^{\frac1d}
    \eqand
    r_+ = \Lp\frac{1}{1-\epsilon}\frac{k}{c_\mu(n-1)}\Rp^{\frac1d}~.
\end{align*}
For all $x\in M$, denote $p_x^- = \mu\bLp\oball(x,r_-)\bRp$, which is upper bounded by $p^- = C_\mu (r_-)^d$ by Ahlfors assumption.
Denote $A_i^-$ the event that $\oball(x_i,r_-)$ contains at least $k$ points from $X_n \setminus \{x_i\}$. Then, since each of the $n-1$ other sample points has a probability $p_x^- \le p^-$ of falling within $\oball(x,r_-)$,
\begin{align*}
    \PP(A_i^-)
    = \int_M \PP(A_i^- | x_i=x)~\d \mu(x)
    = \int_M \phi(n-1,p_x^-,k)~\d \mu(x)
    \le \phi(n-1,p^-,k)~.
\end{align*}
Notice that $(n-1)p^- = (1-\epsilon)k < k$. Then, according to \cref{eq:chernoff1},
\begin{align*}
    \PP(A_i^-)
    \le \exp\bLp(1+\log(1-\epsilon)-1+\epsilon)k\bRp
    \le \exp\Lp-\frac{\epsilon^2}{2}k\Rp~.
\end{align*}
If the graph $G_{r_-}(X_n)$ is not included in $\cG_k(X_n)$, then there exists two sample points $x_i$ and $x_j$ that are at most $r_-$ apart but such that neither of them is one of the $k$ nearest neighbors of the other, and in particular $A_i^-$ holds. Then
\begin{align}
    \label{eq:graph_inclusion1}
    \PP\Lp G_{r_-}(X_n) \not\subset \cG_k(X_n) \Rp
    \le \PP\Lp \bigcup_{i=1}^n A_i^- \Rp
    \le n \exp\Lp-\frac{\epsilon^2}{2}k\Rp~.
\end{align}
For all $x\in M$ denote $p_x^+ = \mu\bLp\cball(x,r_+)\bRp$, which is lower bounded by $p^+ = c_\mu (r_+)^d$ by Ahlfors assumption.
Denote $A_i^+$ the event that $\cball(x_i,r_+)$ contains at most $k-1$ points from $X_n \setminus \{x_i\}$. Then, since each of the $n-1$ other sample points has a probability $p_x^+ \ge p^+$ of falling within $\cball(x,r_+)$,
\begin{align*}
    \PP(A_i^+)
    = \int_M \PP(A_i^+ | x_i=x)~\d \mu(x)
    = \int_M \psi(n-1,p_x^+,k-1)~\d \mu(x)
    \le \psi(n-1,p^+,k)~.
\end{align*}
Notice that $k = (1-\epsilon)(n-1)p^+ < (n-1)p^+$. Then, according to \cref{eq:chernoff2},
\begin{align*}
    \PP(A_i^+)
    \le \exp\Lp\Lp1+\log\Lp\frac{1}{1-\epsilon}\Rp-\frac{1}{1-\epsilon}\Rp k\Rp
    \le \exp\Lp-\frac{\epsilon^2}{2}k\Rp~.
\end{align*}
If the graph $\cG_k(X_n)$ is not included in $G_{r_+}(X_n)$, then there exists two sample points $x_i$ and $x_j$ that are more than $r_-$ apart but such that $x_j$ is one of the $k$ nearest neighbors of $x_i$, and in particular $A_i^+$ holds. Then
\begin{align}
    \label{eq:graph_inclusion2}
    \PP\Lp \cG_k(X_n) \not\subset G_{r_+}(X_n) \Rp
    \le \PP\Lp \bigcup_{i=1}^n A_i^+ \Rp
    \le n \exp\Lp-\frac{\epsilon^2}{2}k\Rp~.
\end{align}
Finally, it follows from \cref{eq:graph_inclusion1,eq:graph_inclusion2} that
\begin{align*}
    \PP\Lp G_{r_-}(X_n) \subset \cG_k(X_n) \subset G_{r_+}(X_n) \Rp
    \ge 1 - 2n \exp\Lp-\frac{\epsilon^2}{2}k\Rp
\end{align*}
which concludes the proof of \cref{prop:graph_ball_to_NN}.

\section{Deterministic Technical Proofs}

This section is devoted to technical proofs involving paths of positive reach and the objects studied in \cref{sec:polygonal_approximation}.

\subsection{Length of Paths}
\label{app:path_length}

In this section we state some useful inequalities deduced from \cref{eq:arcsin_inequality} along with lemmas for comparing the conformal length of paths.
The following quantities are all valid upper bounds of $\arcsin(t)-t$ for all $0<t\le\frac12$, which is deduced from elementary computations.
\begin{align}
    \label{eq:trigo}
    \frac{\arcsin(t)^3}{6}~,\quad
    4\Lp1-\frac3\pi\Rp \arcsin(t)t^2~,\quad
    4\Lp\frac\pi3-1\Rp t^3~,\quad
    \frac18 \Lp \arcsin(2t) - 2t\Rp~.
\end{align}
The lower bounds
\begin{align}
    \label{eq:trigo_lower}
    \arcsin(t) - t
    \ge \arcsin(t)\frac{t^2}{6}
    \ge \frac{t^3}{6}
\end{align}
also hold for all $0\le t\le1$.
Applying the upper bounds from \cref{eq:trigo} to \cref{eq:arcsin_inequality} with $t = \|x-y\| / 2\tau_M$ yields the following inequalities.
\begin{lemma}
\label{lemma:metric_distortion_technical}
If $\|x-y\| \le \tau_M$ then
\begin{align}
    \label{eq:metric_distortion_technical_1}
    \|x-y\|
    &\ge D_M(x,y) - \frac{D_M(x,y)^3}{24\tau_M^2}~,
\end{align}
\begin{align}
    \label{eq:metric_distortion_technical_2}
    D_M(x,y)
    \le \|x-y\| + \Lp\frac\pi3-1\Rp \frac{\|x-y\|^3}{\tau_M^2}
    \le \frac{\pi}{3} \|x-y\|~,
\end{align}
\begin{align}
    \label{eq:metric_distortion_technical_3}
    1 - \frac{\|x-y\|}{D_M(x,y)}
    \le \Lp1-\frac3\pi\Rp \frac{\|x-y\|^2}{\tau_M^2}~,
\end{align}
\begin{align}
    \label{eq:metric_distortion_technical_4}
    D_M(x,y) - \|x-y\|
    \le \frac14 \Lp\tau_M\arcsin\Lp\frac{\|x-y\|}{\tau_M}\Rp - \|x-y\|\Rp~.
\end{align}
\end{lemma}
Note that in \cref{lemma:metric_distortion_technical} the set $M$ may be the whole domain as well as any given curve~$\gamma$ between endpoints $x$ and $y$ and with positive reach, in which case $D_\gamma(x,y) = |\gamma|$.

The following results show that a path with length close to the Euclidean distance between its endpoints must remain close to the straight line at all time. Since $f$ is Lipschitz, it follows that two paths with same endpoints and length close to the Euclidean distance between the endpoints also have similar conformal lengths.
\begin{lemma}
\label{lemma:path_deviation}
Let $x,y\in \RR^\damb$ and $\gamma \in \Gamma(x,y)$ be a path parameterized over $[0,1]$ with constant velocity. Denote $\overline\gamma : t \mapsto (1-t)x+ty$ the straight path from $x$ to $y$. Then
\begin{align}
    \|\gamma-\overline\gamma\|_1
    \le \frac{1}{\sqrt6} \sqrt{|\gamma|^2-|\overline\gamma|^2}~.
\end{align}
\end{lemma}

\begin{proof}
Fix $t\in[0,1]$. Denote $a \le |\gamma|t$ and $b \le |\gamma|(1-t)$ the respective distances from $\gamma(t)$ to $x$ and $y$. Denote $\tau\in\RR$ such that $\overline\gamma(\tau)$ is the orthogonal projection of $\gamma(t)$ onto the line passing through $x$ and $y$. The fact that $\tau$ may not belong to $[0,1]$ does not influence the following. Finally denote $h = \|\gamma(t)-\overline\gamma(\tau)\|$. See \cref{fig:path_deviation} for an illustration of the above.
\begin{figure}[tb]
    \centering
    \includegraphics[width=0.7\textwidth]{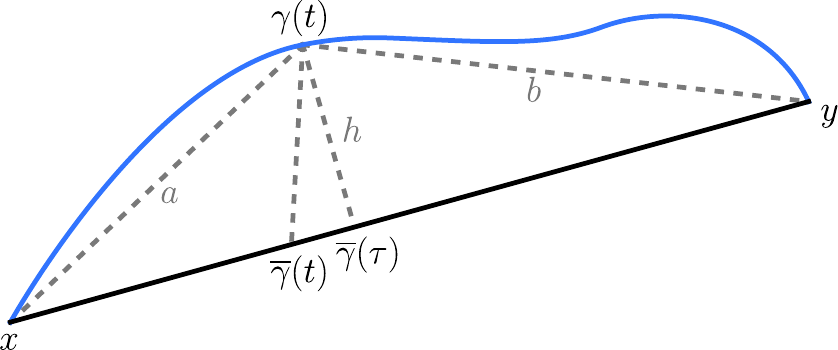}
    \caption{Comparison between the path and the straight line.}
    \label{fig:path_deviation}
\end{figure}
By Pythagoras theorem, it holds that
\begin{align*}
    h^2
    = a^2 - |\overline\gamma|^2\tau^2
    = b^2 - |\overline\gamma|^2(1-\tau)^2
    = \|\gamma(t) - \overline\gamma(t)\|^2 - |\overline\gamma|^2 (t-\tau)^2~.
\end{align*}
In particular,
\begin{align*}
    \|\gamma(t) - \overline\gamma(t)\|^2
    &= (1-t)\bLp a^2 - |\overline\gamma|^2 \tau^2 \bRp
        + t\bLp b^2 - |\overline\gamma|^2 (1-\tau)^2 \bRp
        + |\overline\gamma|^2 (t-\tau)^2\\
    &= (1-t)a^2 + tb^2
        + |\overline\gamma|^2 \bLp (t-1)\tau^2 - t + 2t\tau - t\tau^2 + t^2 - 2t\tau + \tau^2  \bRp\\
    &= (1-t)a^2 + tb^2
        - t(1-t) |\overline\gamma|^2\\
    &\le (1-t)t^2|\gamma|^2 + t(1-t)^2|\gamma|^2 - t(1-t) |\overline\gamma|^2\\
    &= t(1-t) \bLp |\gamma|^2 - |\overline\gamma|^2 \bRp~.
\end{align*}
Integrating the above then yields
\begin{align*}
    \|\gamma-\overline\gamma\|_1
    \le \|\gamma-\overline\gamma\|_2
    \le \sqrt{\int_0^1 t(1-t) \bLp|\gamma|^2 - |\overline\gamma|^2\bRp \d t} 
    = \frac{1}{\sqrt6} \sqrt{|\gamma|^2 - |\overline\gamma|^2}
\end{align*}
which concludes the proof.
\end{proof}

\begin{corollary}
\label{corol:path_deviation_conformal}
Let $x,y\in \RR^\damb$ and $\gamma, \omega \in \Gamma_M(x,y)$.
Then
\begin{align*}
    \Lv \frac{|\gamma|_f}{|\gamma|} - \frac{|\omega|_f}{|\omega|} \Rv
    \le \frac{\kappa}{\sqrt6} \Lp \sqrt{|\gamma|^2-\|x-y\|^2} + \sqrt{|\omega|^2-\|x-y\|^2} \Rp~.
\end{align*}
\end{corollary}

\begin{proof}
Consider parameterizations of $\gamma$ and $\omega$ over $[0,1]$ with constant velocity.
Using in order the $\kappa$-Lipschitz property of $f$, the triangular inequality of $\|\cdot\|_1$ and \cref{lemma:path_deviation} yields
\begin{align*}
    \Lv \frac{|\gamma|_f}{|\gamma|} - \frac{|\omega|_f}{|\omega|} \Rv
    &\le \int_0^1 \bv f(\gamma(t)) - f(\omega(t)) \bv \d t\\
    &\le \kappa \|\gamma-\omega\|_1\\
    &\le \frac{\kappa}{\sqrt6} \Lp \sqrt{|\gamma|^2-\|x-y\|^2} + \sqrt{|\omega|^2-\|x-y\|^2} \Rp
\end{align*}
which concludes the proof.
\end{proof}

\subsection{Regularity of Geodesics}
\label{app:geodesic_reach}

In this section we prove \cref{prop:geodesic_reach} using the results proven in \cref{app:path_length}.
Assume by contradiction that $\tau_{M,f} < \cT_{M,f}$. Then according to \cref{prop:metric_distortion_reach_characterization_conformal} there exists $0 < r < \cT_{M,f} = \frac{\tau_M}{2} \wedge \frac{\fmin}{8\kappa}$ and $x,y\in M$ such that, letting $\gamma \in \Gamma^\star_{M,f}(x,y)$ be a geodesic and $\overline\gamma = t\mapsto (1-t)x+ty$ be the straight path from $x$ to $y$, one has $|\overline\gamma| = \|x-y\| < 2r$ and
\begin{align}
    \label{eq:proof:geodesic_reach_1}
    |\gamma|
    \ge 2r\arcsin\Lp\frac{|\overline\gamma|}{2r}\Rp~.
\end{align}
Let $\omega \in \Gamma^\star_M(x,y)$ be a geodesic \ac{wrt} the induced metric and let us show that $|\omega|_f < |\gamma|_f$, which contradicts the geodesic nature of~$\gamma$. First, rewrite
\begin{align}
    \label{eq:proof:geodesic_reach_2}
    |\omega|_f - |\gamma|_f
    = \Lp\frac{|\omega|}{|\gamma|}-1\Rp |\gamma|_f + |\omega| \Lp \frac{|\omega|_f}{|\omega|} - \frac{|\gamma|_f}{|\gamma|} \Rp~.
\end{align}
The first term in the right-hand side of \cref{eq:proof:geodesic_reach_2} is negative as $|\omega| = D_M(x,y) \le |\gamma|$ hence can be upper bounded by noticing that $|\gamma|_f \ge |\gamma|\fmin$. As for the second term, it can be upper bounded according to \cref{corol:path_deviation_conformal}.
Moreover, recall that $|\overline\gamma| < 2r < \tau_M$, hence \cref{eq:metric_distortion_technical_2} may be rewritten as
\begin{align}
    \label{eq:proof:geodesic_reach_3}
    |\omega|
    \le |\overline\gamma| + \Lp\frac\pi3-1\Rp \frac{|\overline\gamma|^3}{\tau_M^2}
    \le \frac{\pi}{3} |\overline\gamma|~.
\end{align}
Combining these three inequalities yields
\begin{align}
    \label{eq:proof:geodesic_reach_4}
    |\omega|_f - |\gamma|_f
    \le \bLp|\omega|-|\gamma|\bRp \fmin + \frac{\pi\kappa}{3\sqrt6} |\overline\gamma| \Lp \sqrt{|\omega|^2-|\overline\gamma|^2} + \sqrt{|\gamma|^2-|\overline\gamma|^2} \Rp~.
\end{align}
In the following, both terms are upper bounded in a way that factors out $|\gamma|-|\overline\gamma|$.
Using in order \cref{eq:metric_distortion_technical_4}, the fact that $2r \le \tau_M$ and $t\mapsto t\arcsin(|\overline\gamma| / t)$ is non-increasing, and the initial assumption given by \cref{eq:proof:geodesic_reach_1}, it holds that
\begin{align*}
    |\omega| - |\overline\gamma|
    \le \frac14 \Lp\tau_M\arcsin\Lp\frac{|\overline\gamma|}{\tau_M}\Rp - |\overline\gamma|\Rp
    \le \frac14 \Lp2r\arcsin\Lp\frac{|\overline\gamma|}{2r}\Rp - |\overline\gamma|\Rp
    \le \frac14 \bLp|\gamma| - |\overline\gamma|\bRp
\end{align*}
and
\begin{align}
    \label{eq:proof:geodesic_reach_5}
    |\omega| - |\gamma|
    = |\omega| - |\overline\gamma| + |\overline\gamma| - |\gamma|
    \le -\frac34 \bLp|\gamma| - |\overline\gamma|\bRp~.
\end{align}
On the other hand, applying \cref{eq:trigo_lower} to $t = |\overline\gamma| / 2r$, then \cref{eq:proof:geodesic_reach_1},
\begin{align}
    \label{eq:proof:geodesic_reach_6}
    |\overline\gamma|^3
    \le 6(2r)^3 \Lp \arcsin\Lp\frac{|\overline\gamma|}{2r}\Rp - \frac{|\overline\gamma|}{2r} \Rp
    \le 24r^2 \bLp|\gamma| - |\overline\gamma|\bRp~.
\end{align}
Then, using \cref{eq:proof:geodesic_reach_3} again, then \cref{eq:proof:geodesic_reach_6},
\begin{align*}
    |\overline\gamma| \sqrt{|\omega|^2-|\overline\gamma|^2}
    &= |\overline\gamma| \sqrt{|\omega|+|\overline\gamma|} \sqrt{|\omega|-|\overline\gamma|}\\
    &\le |\overline\gamma| \sqrt{\Lp\frac\pi3+1\Rp|\overline\gamma|} \sqrt{\Lp\frac\pi3-1\Rp \frac{|\overline\gamma|^3}{\tau_M^2}}\\
    &= \sqrt{\frac{\pi^2}{9}-1} \frac{|\overline\gamma|^3}{\tau_M}\\
    &\le \sqrt{\frac{\pi^2}{9}-1} \frac{24r^2}{\tau_M} \bLp|\gamma| - |\overline\gamma|\bRp~.
\end{align*}
Hence, since $2r \le \tau_M$,
\begin{align}
    \label{eq:proof:geodesic_reach_7}
    |\overline\gamma| \sqrt{|\omega|^2-|\overline\gamma|^2}
    \le 12\sqrt{\frac{\pi^2}{9}-1}~r \bLp|\gamma| - |\overline\gamma|\bRp~.
\end{align}
Likewise, \cref{eq:proof:geodesic_reach_6} and $|\overline\gamma| \le 2r$ imply that
\begin{align}
    \label{eq:proof:geodesic_reach_8}
    |\overline\gamma| \sqrt{|\gamma|^2-|\overline\gamma|^2}
    = \sqrt{2|\overline\gamma|^3 + |\overline\gamma|^2\bLp|\gamma|-|\overline\gamma|\bRp} \sqrt{|\gamma|-|\overline\gamma|}
    \le \sqrt{48r^2 + (2r)^2} \bLp|\gamma|-|\overline\gamma|\bRp~.
\end{align}
Finally, combining \cref{eq:proof:geodesic_reach_4,eq:proof:geodesic_reach_5,eq:proof:geodesic_reach_7,eq:proof:geodesic_reach_8} yields
\begin{align*}
    |\omega|_f - |\gamma|_f
    \le \Lp -\frac34 \fmin + \frac{\pi}{3\sqrt6} \Lp12\sqrt{\frac{\pi^2}{9}-1} + 2\sqrt{13}\Rp \kappa r \Rp \bLp|\gamma| - |\overline\gamma|\bRp~.
\end{align*}
By assumption,
\begin{align*}
    r
    < \frac{\fmin}{8\kappa}
    \le \Lp12\sqrt{\frac{\pi^2}{9}-1} + 2\sqrt{13}\Rp^{-1} \frac{9\sqrt6}{4\pi} \frac{\fmin}{\kappa}~,
\end{align*}
which concludes that $|\omega|_f - |\gamma|_f$ is negative, hence the contradiction and \cref{prop:geodesic_reach} is proven. The choice of the constant $\frac18$ instead of the one appearing in the right-hand side of the last inequality is purely for cosmetic reasons.

\subsection{Weight Function}
\label{app:weight_props}

In this section we state the Lipschitz property of the weight function defined in \cref{def:weight_function} \ac{wrt} the function $f$ and the endpoints $(x,y)$. The latter is inherited from the Lipschitz nature of $f$. We also quantify the convergence of $w_{f,q}$ to $w_{f,\infty}$ as $q\to\infty$ and then prove \cref{prop:weight_distortion}.
The following three lemmas are obtained through elementary computations that are not detailed here.
Recall that
\begin{align*}
    w_{f,q}(x,y)
    = \frac{\|x-y\|}{2(q-1)} \Lp f(x) + 2\sum_{k=2}^{q-1} f\Lp\frac{q-k}{q-1}x + \frac{k-1}{q-1}y\Rp + f(y) \Rp
\end{align*}
and
\begin{align*}
    w_{f,\infty}(x,y)
    = \|x-y\| \int_0^1 f\bLp(1-t)x+ty\bRp \d t~.
\end{align*}

\begin{lemma}
\label{lemma:weight_lipschitz_endpoints}
Let $f:\RR^\damb \to \RR_+^*$ be a $\kappa$-Lipschitz function and $q\in\{2,\dots,\infty\}$. Then the weight function $w_{f,q}$ satisfies for all $x,y,x',y'\in M$
\begin{align*}
    \Lv\frac{w_{f,q}(x,y)}{\|x-y\|}-\frac{w_{f,q}(x',y')}{\|x'-y'\|}\Rv
    \le \kappa\frac{\|x-x'\| + \|y-y'\|}{2}~.
\end{align*}
\end{lemma}

\begin{lemma}
\label{lemma:weight_lipschitz_f}
Let $f, g : \RR^\damb \to \RR_+^*$ be two functions and $q\in\{2,\dots,\infty\}$. Then for all $x,y\in M$
\begin{align*}
    \bv w_{f,q}(x,y) - w_{g,q}(x,y) \bv
    \le \|x-x\| \|f-g\|_\infty~.
\end{align*}
\end{lemma}

\begin{lemma}
\label{lemma:weight_q_difference}
Let $f:\RR^\damb \to \RR_+^*$ be a $\kappa$-Lipschitz function and $q\ge2$ an integer. Then for all $x,y\in M$
\begin{align*}
    \Lv w_{f,q}(x,y) - w_{f,\infty}(x,y) \Rv
    &\le \frac{\kappa}{4}\frac{\|x-y\|^2}{q-1}~.
\end{align*}
\end{lemma}

\Cref{lemma:weight_distortion_infty} below states the upper bound on the distortion of the weight function when $q=\infty$. The case of finite $q$ is then deduced from \cref{lemma:weight_q_difference}.
Indeed, according to \cref{lemma:weight_distortion_infty,lemma:weight_q_difference}, for all $\|x-y\| \le \tau_{M,f}$ one has
\begin{align*}
    \metricratioOld{w_{f,q}(x,y)}{D_{M,f}(x,y)}
    &\le \Lv\frac{w_{f,\infty}(x,y) - w_{f,q}(x,y)}{D_{M,f}(x,y)}\Rv + \metricratioOld{w_{f,\infty}(x,y)}{D_{M,f}(x,y)}\\
    &\le \frac{\kappa}{4}\frac{\|x-y\|^2}{(q-1)D_{M,f}(x,y)} + \frac{\|x-y\|^2}{16\cT_{M,f}^2}\\
    &\le \frac{\kappa}{4\fmin}\frac{\|x-y\|}{(q-1)} + \frac{\|x-y\|^2}{16\cT_{M,f}^2}
\end{align*}
since $D_{M,f}(x,y) \ge \fmin \|x-y\|$. This concludes \cref{prop:weight_distortion}.

\begin{lemma}
\label{lemma:weight_distortion_infty}
Let $x,y\in M$ such that $\|x-y\| \le \tau_{M,f}$. Then
\begin{align}
    \metricratioOld{w_{f,\infty}(x,y)}{D_{M,f}(x,y)}
    \le \frac{\|x-y\|^2}{16\cT_{M,f}^2}~.
\end{align}
\end{lemma}

\begin{proof}
Let $\gamma\in\Gamma_{M,f}^\star(x,y)$ be a geodesic and $\overline\gamma$ be the straight path from $x$ to $y$. By definition $D_{M,f}(x,y) = |\gamma|_f$, $w_{f,\infty}(x,y) = |\overline\gamma|_f$ and $|\overline\gamma| = \|x-y\|$.
First rewrite
\begin{align}
    \label{eq:proof:weight_distortion_infty_1}
    D_{M,f}(x,y) - w_{f,\infty}(x,y)
    = \Lp1-\frac{|\overline\gamma|}{|\gamma|}\Rp |\gamma|_f + |\overline\gamma| \Lp \frac{|\gamma|_f}{|\gamma|} - \frac{|\overline\gamma|_f}{|\overline\gamma|} \Rp~.
\end{align}
Regarding the first term, as $\|x-y\| \le \tau_{M,f} \le \tau_\gamma$, applying \cref{eq:metric_distortion_technical_3} to the induced metric~$D_\gamma$ yields
\begin{align}
    \label{eq:proof:weight_distortion_infty_2}
    0
    \le 1-\frac{|\overline\gamma|}{|\gamma|}
    = 1-\frac{\|x-y\|}{D_\gamma(x,y)}
    \le \Lp1-\frac3\pi\Rp \frac{\|x-y\|^2}{\tau_\gamma^2}~.
\end{align}
As for the second term, applying \cref{corol:path_deviation_conformal} yields
\begin{align*}
    \Lv \frac{|\gamma|_f}{|\gamma|} - \frac{|\overline\gamma|_f}{|\overline\gamma|} \Rv
    \le \frac{\kappa}{\sqrt6} \sqrt{|\gamma|^2-|\overline\gamma|^2}
    = \frac{\kappa}{\sqrt6} |\gamma| \sqrt{1+\frac{|\overline\gamma|}{|\gamma|}} \sqrt{1-\frac{|\overline\gamma|}{|\gamma|}}
\end{align*}
and using the fact that $|\gamma|_f \ge |\gamma|\fmin$ and $|\overline\gamma| \le |\gamma|$ along with \cref{eq:proof:weight_distortion_infty_2} further implies that
\begin{align}
    \label{eq:proof:weight_distortion_infty_3}
    \Lv \frac{|\gamma|_f}{|\gamma|} - \frac{|\overline\gamma|_f}{|\overline\gamma|} \Rv
    \le \frac{\kappa}{\sqrt6} \frac{|\gamma|_f}{\fmin} \sqrt{2} \sqrt{\Lp1-\frac3\pi\Rp \frac{\|x-y\|^2}{\tau_\gamma^2}}~.
\end{align}
Finally, combining \cref{eq:proof:weight_distortion_infty_1,eq:proof:weight_distortion_infty_2,eq:proof:weight_distortion_infty_3} together with the fact that $\tau_\gamma \ge \cT_{M,f}$ and $\frac{\fmin}{8\kappa} \ge \cT_{M,f}$ yields
\begin{align*}
    \metricratioOld{w_{f,\infty}(x,y)}{D_{M,f}(x,y)}
    &\le \Lp \Lp1-\frac3\pi\Rp\frac{1}{\tau_\gamma^2} + \sqrt{\frac13-\frac1\pi}\frac{\kappa}{\fmin\tau_\gamma} \Rp \|x-y\|^2\\
    &\le \Lp 1-\frac3\pi + \frac18\sqrt{\frac13-\frac1\pi} \Rp \frac{\|x-y\|^2}{\cT_{M,f}^2}\\
    &\le \frac{\|x-y\|^2}{16\cT_{M,f}^2}
\end{align*}
where the last inequality is purely cosmetic and concludes the proof.
\end{proof}

\subsection{Discrete Approximation of the Conformal Metric}
\label{app:proof:geodesic_approx}

In this section we prove \cref{thm:geodesic_approx}.
Let $X \subset M$ be a point cloud and denote $\rho = \hausdorff(M,X)$. Assume that $4\rho \le r \le \cT_{M,f}$ and let $x,y\in M$.

\proofsubparagraph{Lower bound}
Let $(x_0, \dots, x_K)$ be a polygonal path in $G_r(X\cup\{x,y\})$ from $x=x_0$ to $y=x_K$ achieving the infimum in \cref{def:estimator_conform}.
According to \cref{prop:weight_distortion},
\begin{align*}
    \widehat D_{X,f}(x,y)
    = \sum_{k=1}^{K-1} w_{f,q}(x_k, x_{k+1})
    \ge \sum_{k=1}^{K-1} \bLp 1 - \delta_q(\|x_k-x_{k+1}\|) \bRp D_{M,f}(x_k, x_{k+1})~.
\end{align*}
Using the fact that $\delta_q$ is non-decreasing and that $\|x_k-x_{k+1}\| \le r$ for all $k$, along with the triangular inequality of $D_{M,f}$, it follows that
\begin{align}
    \label{eq:proof:geodesic_approx_1}
    \widehat D_{X,f}(x,y)
    \ge \bLp 1 - \delta_q(r) \bRp D_{M,f}(x,y)~.
\end{align}

\proofsubparagraph{Local upper bound}
If $\|x-y\| \le r$, the edge $(x,y)$ belongs to $G_r(X\cup\{x,y\})$, hence according to \cref{prop:weight_distortion},
\begin{align}
    \label{eq:proof:geodesic_approx_2}
    \widehat D_{X,f}(x,y)
    \le w_{f,q}(x,y)
    \le \bLp 1 + \delta_q(r) \bRp D_{M,f}(x,y)~.
\end{align}

\proofsubparagraph{Global upper bound}
Assume that $\|x-y\| \ge r$. Let $\gamma \in \Gamma_{M,f}^\star(x,y)$ be a 
geodesic and decompose it into $K = \lceil|\gamma|/(r-2\rho)\rceil$ sections of equal arc-length $l = |\gamma|/K$. Denote $(x=x_0, x_1, \dots, x_K=y)$ the intermediate points resulting from this decomposition and $(x=x'_0, x'_1, \dots, x'_K=y)$ their respective nearest neighbors in the point cloud $X\cup\{x,y\}$.
Then
\begin{align*}
    1 \le \frac{|\gamma|}{r-2\rho} \le K \le \frac{2|\gamma|}{r-2\rho}
    \eqand
    \frac{r}{4} \le \frac{r}{2}-\rho \le l \le r-2\rho~.
\end{align*}
In particular, $\|x'_k-x'_{k+1}\| \le \|x'_k-x_k\| + \|x_k-x_{k+1}\| + \|x_{k+1}-x'_{k+1}\| \le \rho + l + \rho \le r$ so that $(x'_0,\dots,x'_K)$ is indeed a path in $G_r(X\cup\{x,y\})$.
Denote for short
\begin{align*}
    &\epsilon_k = x'_k-x_k~,
        && w_k = w_{f,q}(x_k, x_{k+1})~,
        && w'_k = w_{f,q}(x'_k, x'_{k+1})~,\\
    & u_k = \frac{x_k-x_{k+1}}{\|x_k-x_{k+1}\|}~,
        && \overline w_k = \frac{w_{f,q}(x_k, x_{k+1})}{\|x_k-x_{k+1}\|}~,
        && \overline w'_k = \frac{w_{f,q}(x'_k, x'_{k+1})}{\|x'_k-x'_{k+1}\|}~.
\end{align*}
Let us show that the weight $\sum_k w_k'$ of the polygonal path is not much greater than the conformal distance $D_{M,f}(x,y)$. First, rewrite
\begin{align}
    \label{eq:proof:geodesic_approx_3}
    w'_k
    = w_k + \|x'_k-x'_{k+1}\| \Lp\overline w'_k-\overline w_k\Rp + \bLp\|x'_k-x'_{k+1}\| - \|x_k-x_{k+1}\|\bRp \overline w_k~.
\end{align}
In order to upper bound the last term in \cref{eq:proof:geodesic_approx_3}, we use the following inequality from the proof of Theorem 2.1 in \cite{aaronConvergenceRatesEstimators2018}:
\begin{align*}
    \|x'_k-x'_{k+1}\|
    \le \|x_k-x_{k+1}\| + \langle u_k, \epsilon_k-\epsilon_{k+1}\rangle + \frac12\frac{\|\epsilon_k-\epsilon_{k+1}\|^2}{\|x_k-x_{k+1}\|}~,
\end{align*}
which implies that
\begin{align*}
    \bLp\|x'_k-x'_{k+1}\| - \|x_k-x_{k+1}\|\bRp \overline w_k
    \le \langle \overline w_k u_k, \epsilon_k-\epsilon_{k+1} \rangle + \frac{w_k}{2} \frac{\|\epsilon_k-\epsilon_{k+1}\|^2}{\|x_k-x_{k+1}\|^2}
\end{align*}
and from \cref{eq:proof:geodesic_approx_3} it follows that
\begin{align}
    \label{eq:proof:geodesic_approx_4}
    w'_k
    \le 
    \underbrace{w_k}_{\textrm{(i)}}
    + \underbrace{\|x'_k-x'_{k+1}\| \Lp\overline w'_k-\overline w_k\Rp}_{\textrm{(ii)}}
    + \underbrace{\langle \overline w_k u_k, \epsilon_k-\epsilon_{k+1} \rangle}_{\textrm{(iii)}}
    + \underbrace{\frac{w_k}{2} \frac{\|\epsilon_k-\epsilon_{k+1}\|^2}{\|x_k-x_{k+1}\|^2}}_{\textrm{(iv)}}~.
\end{align}
Let us now upper bound all four terms in this upper bound.
\begin{enumerate}[(i)]
    \item \Cref{prop:weight_distortion} ensures that
    \begin{align}
        \label{eq:proof:geodesic_approx_5}
        w_k
        \le \bLp 1 + \delta_q(\|x_k-x_{k+1}\|) \bRp D_{M,f}(x_k, x_{k+1})
        \le \bLp 1 + \delta_q(r) \bRp D_{M,f}(x_k, x_{k+1})~.
    \end{align}
    \item \Cref{lemma:weight_lipschitz_endpoints} ensures that
    \begin{align*}
        \|x'_k-x'_{k+1}\| \Lp\overline w'_k-\overline w_k\Rp
        &\le \|x'_k-x'_{k+1}\| \frac{\kappa}{2}\Lp \|x'_k-x_k\| + \|x'_{k+1}-x_{k+1}\| \Rp\\
        &\le (l+2\rho) \kappa\rho~,
    \end{align*}
    which implies since $\rho \le \frac{r}{4} \le l$ and $D_{M,f}(x_k, x_{k+1}) \ge l\fmin$ that
    \begin{align}
        \label{eq:proof:geodesic_approx_6}
        \|x'_k-x'_{k+1}\| \Lp\overline w'_k-\overline w_k\Rp
        \le 3\frac{\kappa}{\fmin}\rho~D_{M,f}(x_k, x_{k+1})~.
    \end{align}
    \item Recall that $\epsilon_0 = \epsilon_K = 0$. Then when summing over $k$, the terms can be rearranged so that
    \begin{align*}
        \sum_{k=0}^{K-1} \langle \overline w_k u_k, \epsilon_k-\epsilon_{k+1} \rangle
        &= \sum_{k=1}^{K-1} \Lan \overline w_k u_k - \overline w_{k-1} u_{k-1}, \epsilon_k \Ran\\
        &= \sum_{k=1}^{K-1} \bLp \overline w_k \Lan u_k - u_{k-1}, \epsilon_k \Ran + (\overline w_k - \overline w_{k-1}) \langle u_{k-1}, \epsilon_k \rangle \bRp~.
    \end{align*}
    According to \cref{lemma:path_direction_lipschitz}, $\|u_k-u_{k+1}\| \le \frac{1}{\tau_\gamma} \|x_k-x_{k+1}\|$ since $r-2\rho \le \cT_{M,f} \le \tau_\gamma$.
    According to \cref{lemma:weight_lipschitz_endpoints}, $\overline w_k - \overline w_{k-1} \le \frac{\kappa}{2} \bLp\|x_k-x_{k+1}\|+\|x_{k+1}-x_{k+2}\|\bRp \le \kappa l$.
    Then, using Cauchy-Schwartz inequality and $D_{M,f}(x_k, x_{k+1}) \ge l\fmin$,
    \begin{align*}
        \overline w_k \Lan u_k - u_{k-1}, \epsilon_k \Ran + (\overline w_k - \overline w_{k-1}) \langle u_{k-1}, \epsilon_k \rangle
        &\le \frac{\overline w_k}{\tau_\gamma}\|x_k-x_{k+1}\|\rho + \kappa l\rho\\
        &\le \frac{w_k}{\tau_\gamma}\rho + \frac{\kappa}{\fmin} D_{M,f}(x_k, x_{k+1})\rho~.
    \end{align*}
    Adding the analogous term corresponding to $k=0$,
    \begin{align}
        \label{eq:proof:geodesic_approx_7}
        \sum_{k=0}^{K-1} \langle \overline w_k u_k, \epsilon_k-\epsilon_{k+1} \rangle
        \le \sum_{k=0}^{K-1} \Lp \frac{w_k}{\tau_\gamma} + \frac{\kappa}{\fmin} D_{M,f}(x_k, x_{k+1}) \Rp \rho~.
    \end{align}
    \item Since $\|x_k-x_{k+1}\| \le D_\gamma(x_k,x_{k+1}) = l \le \tau_\gamma$ and $r \le 4l$, applying \cref{eq:metric_distortion_technical_1} to the metric~$D_\gamma$ yields
    \begin{align*}
        \|x_k-x_{k+1}\|
        \ge l - \frac{l^3}{24\tau_\gamma^2}
        \ge \frac{23}{24} l
        \ge \frac{23}{96} r~,
    \end{align*}
    hence
    \begin{align}
        \label{eq:proof:geodesic_approx_8}
        \frac{w_k}{2} \frac{\|\epsilon_k-\epsilon_{k+1}\|^2}{\|x_k-x_{k+1}\|^2}
        \le \frac{w_k}{2} \Lp\frac{96}{23}\Rp^2 \frac{4\rho^2}{r^2}~.
    \end{align}
\end{enumerate}
Finally, since $(x_0,\dots,x_K)$ are intermediate points on the geodesic $\gamma$ from $x$ to $y$, $D_{M,f}(x,y) = \sum_{k=0}^{K-1} D_{M,f}(x_k,x_{k+1})$. Together with \cref{eq:proof:geodesic_approx_4,eq:proof:geodesic_approx_5,eq:proof:geodesic_approx_6,eq:proof:geodesic_approx_7,eq:proof:geodesic_approx_8}, this yields
\begin{align*}
    \sum_{k=0}^{K-1} w'_k
    \le \Lp 1 + \delta_q(r) + 4\frac{\kappa}{\fmin}\rho \Rp D_{M,f}(x,y) + \Lp \frac{\rho}{\tau_\gamma} + 2\Lp\frac{96}{23}\Rp^2\frac{\rho^2}{r^2} \Rp \sum_{k=0}^{K-1} w_k
\end{align*}
In particular, recall that $\frac{\kappa}{\fmin} \le \frac{1}{8\cT_{M,f}}$ and $r\le\cT_{M,f}$, hence $\delta_q(r) \le \frac{9}{64}$ and \cref{eq:proof:geodesic_approx_5} further implies that $w_k \le \frac{73}{64} D_{M,f}(x_k, x_{k+1})$, so that
\begin{align*}
    \widehat D_{X,f}(x,y)
    \le \sum_{k=0}^{K-1} w'_k
    \le \Lp1 + \delta_q(r) + \frac{105}{64}\frac{\rho}{\cT_{M,f}} + \frac{73}{32}\Lp\frac{96}{23}\Rp^2\frac{\rho^2}{r^2}\Rp D_{M,f}(x,y)~.
\end{align*}
Upper bounding $\frac{105}{64} \le 2$ and $\frac{73}{32}\Lp\frac{96}{23}\Rp^2 \le 40$ to get cleaner constants, it finally holds that
\begin{align}
    \label{eq:proof:geodesic_approx_9}
    \widehat D_{X,f}(x,y)
    \le \Lp1 + \delta_q(r) + 2\frac{\rho}{\cT_{M,f}} + 40\frac{\rho^2}{r^2}\Rp D_{M,f}(x,y)~.
\end{align}

\proofsubparagraph{Conclusion}
According to the lower bound \cref{eq:proof:geodesic_approx_1} and to the upper bounds \cref{eq:proof:geodesic_approx_2,eq:proof:geodesic_approx_9} depending on the case, the distortion between the polygonal metric and the original metric is upper bounded as
\begin{align*}
    l_{\infty,M}\bLp\widehat D_{X,f}| D_{M,f}\bRp
    &\le \delta_q(r) + 2\frac{\rho}{\cT_{M,f}} + 40\frac{\rho^2}{r^2}\\
    &= \frac{\kappa}{4\fmin}\frac{r}{q-1} + \frac{r^2}{16\cT_{M,f}^2} + 2\cdot\frac{r}{4\cT_{M,f}}\cdot\frac{4\rho}{r} + 40\frac{\rho^2}{r^2}\\
    &\le \frac{1}{32\cT_{M,f}}\frac{r}{q-1} + \frac{r^2}{8\cT_{M,f}^2} + 56\frac{\rho^2}{r^2}
\end{align*}
where we used the inequality $2ab \le a^2+b^2$. This concludes the proof of \cref{thm:geodesic_approx}.

\subsection{Approximation of the Conformal Factor}
\label{app:proof:estimator_convergence_both_errors}

In this section we prove \cref{lemma:estimator_convergence_both_errors}.
Fix $x,y\in M$ and denote $\widehat D_g = \widehat D_{X,g}(x,y)$, $\widehat D_f = \widehat D_{X,f}(x,y)$ and $D = D_{M,f}(x,y)$ for short. Denote $G$ the graph considered for the paths in both $\widehat D_g$ and $\widehat D_f$ and $q$ the resolution.
First, notice that the assumption that $\|g-f\|_\infty \le \frac12\fmin$ implies the lower bound $g_{\min} = \frac12\fmin$ on $g$.
For any edge $(u,v)$ in $G$, \cref{lemma:weight_lipschitz_f} yields
\begin{align}
    \label{eq:proof:estimator_convergence_both_errors_1}
    \bv w_{g,q}(u,v) - w_{f,q}(u,v) \bv
    \le \|u-v\|\|g-f\|_\infty~.
\end{align}
Then, since $w_{g,q}(u,v) \ge \|u-v\| g_{\min}$, summing \cref{eq:proof:estimator_convergence_both_errors_1} over a path that realizes $\widehat D_g$ yields
\begin{align}
    \label{eq:proof:estimator_convergence_both_errors_2}
    \widehat D_f
    \le \Lp 1 + \frac{\|g-f\|_\infty}{g_{\min}} \Rp \widehat D_g
    \le 2 \widehat D_g~.
\end{align}
On the other hand, since $w_{f,q}(u,v) \ge \|u-v\| \fmin$, summing \cref{eq:proof:estimator_convergence_both_errors_1} over a path that realizes $\widehat D_f$ yields
\begin{align*}
    \widehat D_g
    \le \Lp 1 + \frac{\|g-f\|_\infty}{\fmin} \Rp \widehat D_f~,
\end{align*}
and it follows that
\begin{align}
    \label{eq:proof:estimator_convergence_both_errors_3}
    \Lv \widehat D_g - \widehat D_f \Rv
    \le \max\Lp \frac{\|g-f\|_\infty}{g_{\min}} \widehat D_g , \frac{\|g-f\|_\infty}{\fmin} \widehat D_f \Rp
    = 2\frac{\|g-f\|_\infty}{\fmin} \widehat D_g~.
\end{align}
Then, \cref{eq:proof:estimator_convergence_both_errors_2,eq:proof:estimator_convergence_both_errors_3} imply that 
\begin{align*}
    \Lv \widehat D_g - D \Rv
    &\le \Lv \widehat D_g - \widehat D_f \Rv + \Lv \widehat D_f - D \Rv\\
    &\le 2\frac{\|g-f\|_\infty}{\fmin} \widehat D_g + \ell_{\infty,M}\bLp\widehat D_{X,f}, D_{M,f}\bRp \bLp\widehat D_f \vee D\bRp\\
    &\le 2\frac{\|g-f\|_\infty}{\fmin} \widehat D_g + 2\ell_{\infty,M}\bLp\widehat D_{X,f}, D_{M,f}\bRp \bLp\widehat D_g \vee D\bRp~.
\end{align*}
Taking the supremum over $x$ and $y$ finally yields
\begin{align*}
    \ell_{\infty,M}\Lp\widehat D_{X,g}, D_{M,f}\Rp
    \le 2\Lp \frac{\|g-f\|_\infty}{\fmin} + \ell_{\infty,M}\bLp\widehat D_{X,f},D_{M,f}\bRp \Rp
\end{align*}
which concludes the proof of \cref{lemma:estimator_convergence_both_errors}.

\section{Probabilistic Technical Proofs}

\subsection{Convergence on the Ball Graph}
\label{app:proof::convergence_ball_graph}

In this section we prove \cref{thm:convergence_ball_graph}. In the following we say that something is true for $n$ large enough when it is true for all $n$ larger than some constant depending on $L_\mu$, $\cT_{M,f}$ and~$d$.
Assume that $0 < r \le \cT_{M,f}$ and $q\in\{2,\dots,\infty\}$. Then
\begin{equation*}
    \EE\Lb \ell_{\infty,M}\bLp\widehat D_{X_n,f}, D_{M,f}\bRp \Rb
    \le \PP\bLp 4\hausdorff(M,X_n) > r \bRp + \EE\Lb \1_{(4\hausdorff(M,X_n) \le r)} l_{\infty,M}\bLp\widehat D_{X_n,f}| D_{M,f}\bRp \Rb
\end{equation*}
since $\ell_{\infty,M}$ is upper bounded both by $1$ and by $l_{\infty,M}$ according to \cref{eq:loss_comparison}. According to \cref{lemma:hausdorff_proba_bound},
\begin{align}
    \label{eq:proof:convergence_ball_graph_1}
    \PP\bLp 4\hausdorff(M,X_n) > r \bRp
    \le \frac{16^d}{c_\mu r^d} e^{-n c_\mu \Lp\frac{r}{8}\Rp^d}
    = \Lp\frac{16L_\mu}{r}\Rp^d e^{-n \Lp\frac{r}{8L_\mu}\Rp^d}~.
\end{align}
On the other hand, assuming that $n\ge8$, \cref{thm:geodesic_approx,corol:hausdorff_expectation} yield
\begin{align}
    \label{eq:proof:convergence_ball_graph_2}
    \EE\Lb \1_{(4\hausdorff(M,X_n) \le r)} \ell_{\infty,M}\bLp\widehat D_{X_n,f}, D_{M,f}\bRp \Rb
    \le \frac{r}{32(q-1)\cT_{M,f}} + \frac{r^2}{8\cT_{M,f}^2} + 448\frac{L_\mu^2}{r^2} \Lp\frac{\log(n)}{n}\Rp^{\frac2d}
\end{align}
so that
\begin{align}
    \label{eq:proof:convergence_ball_graph_3}
    \EE\Lb \ell_{\infty,M}\bLp\widehat D_{X_n,f}, D_{M,f}\bRp \Rb
    \le A_n(r) + B_n(r,q)
\end{align}
where $A_n(r)$ and $B_n(r,q)$ respectively denote the right-hand sides in \cref{eq:proof:convergence_ball_graph_1} and in \cref{eq:proof:convergence_ball_graph_2}.
We now consider the two cases for the choice of the parameters $r$ and $q$ specified in \cref{thm:convergence_ball_graph}.

\proofsubparagraph{Large resolution}
Let $r = 8\sqrt{L_\mu\cT_{M,f}} (\log(n)/n)^{1/2d}$, which is indeed smaller than $\cT_{M,f}$ for $n$ large enough. Let $q \ge 1 + 4\frac{\cT_{M,f}}{r}$. Then
\begin{align*}
    A_n(r)
    = 2^d\Lp\frac{L_\mu}{\cT_{M,f}}\Rp^{\frac{d}{2}} \sqrt{\frac{n}{\log(n)}} \exp\Lp-\Lp\frac{\cT_{M,f}}{L_\mu}\Rp^{\frac{d}{2}}\sqrt{n\log(n)}\Rp
    \le \frac12\frac{L_\mu}{\cT_{M,f}} \Lp\frac{\log(n)}{n}\Rp^{\frac{1}{d}}
\end{align*}
for $n$ large enough and
\begin{align*}
    B_n(r,q)
    \le B_n\Lp1 + 4\frac{\cT_{M,f}}{r} , r\Rp
    = \Lp\frac12+8+7\Rp \frac{L_\mu}{\cT_{M,f}} \Lp\frac{\log(n)}{n}\Rp^{\frac1d}~,
\end{align*}
so that according to \cref{eq:proof:convergence_ball_graph_3},
\begin{align*}
    \EE\Lb \ell_{\infty,M}\bLp\widehat D_{X_n,f}, D_{M,f}\bRp \Rb
    \le 16\frac{L_\mu}{\cT_{M,f}} \Lp\frac{\log(n)}{n}\Rp^{\frac1d}
\end{align*}
which concludes the proof of \cref{eq:thm:convergence_ball_graph_1}.

\proofsubparagraph{Small resolution}
Let $r = 8L_\mu^{2/3}\cT_{M,f}^{1/3} (\log(n)/n)^{2/3d}$, which is indeed smaller than $\cT_{M,f}$ for $n$ large enough. Let $q = 2$. Then
\begin{align*}
    A_n(r)
    &= 2^d\Lp\frac{L_\mu}{\cT_{M,f}}\Rp^{\frac{d}{3}} \Lp\frac{n}{\log(n)}\Rp^{\frac{2}{3}} \exp\Lp-\Lp\frac{\cT_{M,f}}{L_\mu}\Rp^{\frac{d}{3}}n^{\frac13}\log(n)^{\frac23}\Rp\\
    &\le \frac12\Lp\frac{L_\mu}{\cT_{M,f}}\Rp^{\frac23} \Lp\frac{\log(n)}{n}\Rp^{\frac{2}{3d}}
\end{align*}
for $n$ large enough and
\begin{align*}
    B_n(q,r)
    &= \Lp\frac14+7\Rp \Lp\frac{L_\mu}{\cT_{M,f}}\Rp^{\frac23} \Lp\frac{\log(n)}{n}\Rp^{\frac{2}{3d}} + 8 \Lp\frac{L_\mu}{\cT_{M,f}}\Rp^{\frac43} \Lp\frac{\log(n)}{n}\Rp^{\frac{4}{3d}}\\
    &\le \frac{15}{2} \Lp\frac{L_\mu}{\cT_{M,f}}\Rp^{\frac23} \Lp\frac{\log(n)}{n}\Rp^{\frac{2}{3d}}
\end{align*}
for $n$ large enough, so that according to \cref{eq:proof:convergence_ball_graph_3},
\begin{align*}
    \EE\Lb \ell_{\infty,M}\bLp\widehat D_{X_n,f}, D_{M,f}\bRp \Rb
    \le 8\Lp\frac{L_\mu}{\cT_{M,f}}\Rp^{\frac23} \Lp\frac{\log(n)}{n}\Rp^{\frac{2}{3d}}
\end{align*}
which concludes the proof of \cref{eq:thm:convergence_ball_graph_2}.

\subsection{Convergence on the Nearest-Neighbors Graph}
\label{app:proof:convergence_NN_graph}

In this section we prove \cref{thm:convergence_NN_graph} using \cref{thm:convergence_ball_graph,prop:graph_ball_to_NN}.
Consider $k=\big\lceil\sqrt{n\log(n)}~\big\rceil$, $q=\lceil n^{1/4}\rceil$ and let
\begin{align*}
    r_- = \Lp\frac{k}{2C_\mu(n-1)}\Rp^{\frac1d}
    \eqand
    r_+ = \Lp\frac{2k}{c_\mu(n-1)}\Rp^{\frac1d}~,
\end{align*}
so that according to \cref{prop:graph_ball_to_NN}, the event
\begin{align*}
    A_1
    \eqdef \La G_{r_-}(X_n) \subset \cG_k(X_n) \subset G_{r_+}(X_n) \Ra
\end{align*}
satisfies
\begin{align}
    \label{eq:proof:convergence_ball_to_NN_1}
    \PP\bLp\overline{A_1}\bRp
    \le 2n \exp\Lp-\frac18\sqrt{n\log(n)}\Rp~.
\end{align}
Denote $\widehat D_{r_-}$, $\widehat D_{r_+}$ and $\widehat D_k$ the three different instances of $\widehat D_{X_n,f}$ defined for the $r_-$-ball graph, the $r_+$-ball graph and the $k$-NN graph respectively, all with resolution $q$. Also denote ${D = D_{M,f}}$ for short.
Under the event $A_1$, for all $x,y \in X_n$ it holds that
\begin{align*}
    \widehat D_{r_+}(x,y)
    \le \widehat D_k(x,y)
    \le \widehat D_{r_-}(x,y)
\end{align*}
as the graphs are equipped with the same weight function $w_{f,q}$ and the infimum over paths increases when fewer paths are allowed due to the graph being smaller.
It follows that
\begin{align*}
    l_{\infty,X_n}\bLp\widehat D_k| D\bRp
    &= \sup_{x\ne y\in M} \Lv \frac{D(x,y) - \widehat D_k(x,y)}{D(x,y)} \Rv\\
    &\le \sup_{x\ne y\in M} \Lv \frac{D(x,y) - \widehat D_{r_-}(x,y)}{D(x,y)} \Rv \vee \Lv \frac{D(x,y) - \widehat D_{r_+}(x,y)}{D(x,y)} \Rv\\
    &= l_{\infty,X_n}\bLp\widehat D_{r_-}| D\bRp \vee l_{\infty,X_n}\bLp\widehat D_{r_+}| D\bRp~.
\end{align*}
Under the event
\begin{align*}
    A_2
    \eqdef \La \ell_{\infty,X}\bLp\widehat D_{r_-}, D\bRp \le \frac12 \Ra \cap \La \ell_{\infty,X_n}\bLp\widehat D_{r_+}, D\bRp \le \frac12 \Ra~,
\end{align*}
it then holds according to \cref{eq:loss_comparison} that
\begin{align*}
    \ell_{\infty,X_n}\bLp\widehat D_k, D\bRp
    &\le l_{\infty,X_n}\bLp\widehat D_k| D\bRp\\
    &\le l_{\infty,X_n}\bLp\widehat D_{r_-}| D\bRp \vee l_{\infty,X_n}\bLp\widehat D_{r_+}| D\bRp\\
    &\le 2\ell_{\infty,X_n}\bLp\widehat D_{r_-}, D\bRp + 2\ell_{\infty,X_n}\bLp\widehat D_{r_+}, D\bRp~.
\end{align*}
This implies that
\begin{align}
    \label{eq:proof:convergence_ball_to_NN_2}
    \EE\Lb\1_{A_1\cap A_2}~\ell_{\infty,X_n}\bLp\widehat D_k, D\bRp \Rb
    \le 2\EE\Lb \ell_{\infty,X_n}\bLp\widehat D_{r_-}, D\bRp \Rb + 2\EE\Lb \ell_{\infty,X_n}\bLp\widehat D_{r_+}, D\bRp \Rb~.
\end{align}
The event $A_1$ is already known to have high probability. Regarding $A_2$, Markov inequality yields
\begin{align}
    \label{eq:proof:convergence_ball_to_NN_3}
    \PP\bLp\overline{A_2}\bRp
    \le 2\EE\Lb \ell_{\infty,X_n}\bLp\widehat D_{r_-}, D\bRp \Rb + 2\EE\Lb \ell_{\infty,X_n}\bLp\widehat D_{r_+}, D\bRp \Rb~.
\end{align}
Since $r_- \asymp r_+ \asymp (\log(n)/n)^{1/2d}$ when $n$ is large and $q = \lceil n^{1/4}\rceil \gtrsim 1/r_+ \ge 1/r_-$, there exists according to \cref{thm:convergence_ball_graph} a constant $c$ depending on $L_\mu$ and $\cT_{M,f}$ such that
\begin{align}
    \label{eq:proof:convergence_ball_to_NN_4}
    \EE\Lb \ell_{\infty,X_n}\bLp\widehat D_{r_-}, D\bRp \Rb + \EE\Lb \ell_{\infty,X_n}\bLp\widehat D_{r_+}, D\bRp \Rb
    \le c \Lp\frac{\log(n)}{n}\Rp^{\frac{1}{d}}~.
\end{align}
Finally, as $\ell_{\infty,X_n} \le 1$, combining \cref{eq:proof:convergence_ball_to_NN_1,eq:proof:convergence_ball_to_NN_2,eq:proof:convergence_ball_to_NN_3,eq:proof:convergence_ball_to_NN_4} yields
\begin{align*}
    \EE\Lb \ell_{\infty,X_n}\bLp\widehat D_k, D\bRp \Rb
    &\le \EE\Lb\1_{A_1\cap A_2}~\ell_{\infty,X_n}\bLp\widehat D_k, D\bRp \Rb + \PP\bLp\overline{A_2}\bRp + \PP\bLp\overline{A_1}\bRp\\
    &\le 4\EE\Lb \ell_{\infty,X_n}\bLp\widehat D_{r_-}, D\bRp \Rb + 4\EE\Lb \ell_{\infty,X_n}\bLp\widehat D_{r_+}, D\bRp \Rb + 2n e^{-\frac18\sqrt{n\log(n)}}\\
    &\le C \Lp\frac{\log(n)}{n}\Rp^{\frac{1}{d}}
\end{align*}
where $C$ is a constant depending on $L_\mu$ and $\cT_{M,f}$. This concludes the proof of \cref{thm:convergence_NN_graph}.

\subsection{Minimax Lower Bound}
\label{app:minimax}

In this section we prove \cref{thm:minimax_lower_bound} using Le Cam's method \cite{yuAssouadFanoCam1997}.
We first adapt this standard result to our context.
\begin{lemma}
\label{lemma:lecam}
Let $n\ge2$ and $\cM$ be some set of distributions over $\RR^\damb$ such that there exists two measures $\mu_1,\mu_2\in\cM$ with respective supports $M_1$ and $M_2$ and at total variation distance $\dtv(\mu_1,\mu_2) \le \frac1n$ from each other.
Then, for all $x,y \in M_1 \cap M_2$ the minimax risk associated with the estimation of the induced metric may be lower bounded by the distortion between the two metrics $D_{M_1}$ and $D_{M_2}$ at endpoints $x$ and $y$.
\begin{align}
    \label{eq:lemma:lecam}
    \inf_{\hat D \in \mathfrak D_n} \sup_{\mu\in\cM} \EE_{X \sim \mu^{\otimes n}}\Lb \ell_{\infty,M_\mu}\bLp\widehat D_X, D_{M_\mu}\bRp \Rb
    \ge \frac18 \metricratio{D_{M_1}(x,y)}{D_{M_2}(x,y)}
\end{align}
where we recall $\mathfrak D_n$ is the set of possible estimators defined in \cref{thm:minimax_lower_bound} and $M_\mu$ denotes the support of $\mu$.
\end{lemma}

\begin{proof}
Denote $\cR_n$ the left-hand side of \cref{eq:lemma:lecam} to be lower bounded.
The supremum over all measures in $\cM$ is lower bounded by the supremum over the two measures $\mu_1$ and $\mu_2$ and the loss $\ell_{\infty,M_\mu}$ by the distortion between endpoints $x$ and $y$. Finally, the supremum is lower bounded by the average. These inequalities yield
\begin{align*}
    \cR_n
    &\ge \inf_{\widehat D} \sup_{\mu \in \{\mu_1,\mu_2\}} \EE_{\mu^{\otimes n}}\Lb\ell_{\infty,M_\mu}\bLp\widehat D_X, D_M\bRp\Rb\\
    &\ge \inf_{\widehat D} \sup_{\mu \in \{\mu_1,\mu_2\}} \EE_{\mu^{\otimes n}}\Lb\metricratio{\widehat D_X(x,y)}{D_{M_\mu}(x,y)}\Rb\\
    &\ge \frac12 \inf_{\widehat D}\La
        \EE_{\mu_1^{\otimes n}}\Lb\metricratio{\widehat D_X(x,y)}{D_{M_1}(x,y)}\Rb
        + \EE_{\mu_2^{\otimes n}}\Lb\metricratio{\widehat D_X(x,y)}{D_{M_2}(x,y)}\Rb
    \Ra\\
    &\ge \frac12 \inf_{\widehat D}
        \EE_{\mu_1^{\otimes n}}\Lb\Lp \metricratio{\widehat D_X(x,y)}{D_{M_1}(x,y)}
        + \metricratio{\widehat D_X(x,y)}{D_{M_2}(x,y)} \Rp\Lp1\wedge\frac{d\mu_2^{\otimes n}}{d\mu_1^{\otimes n}}\Rp\Rb
\end{align*}
and it can be shown using elementary computations that for all $a,b>0$,
\begin{align*}
    \min_{\delta\ge0} \Lp \metricratio{\delta}{a} + \metricratio{\delta}{b} \Rp
    = \metricratio{a}{b}~,
\end{align*}
hence
\begin{align*}
    \cR_n
    \ge \frac12 \EE_{\mu_1^{\otimes n}}\Lb \metricratio{D_{M_1}(x,y)}{D_{M_2}(x,y)}
        \Lp1\wedge\frac{d\mu_2^{\otimes n}}{d\mu_1^{\otimes n}}\Rp \Rb~.
\end{align*}
Moreover, using the properties of the total variation distance along with the assumption that $\dtv(\mu_1,\mu_2) \le \frac1n$ and $n\ge2$,
\begin{align*}
    \EE_{\mu_1^{\otimes n}}\Lb1\wedge\frac{d\mu_2^{\otimes n}}{d\mu_1^{\otimes n}}\Rb
    = 1 - \dtv(\mu_1^{\otimes n}, \mu_2^{\otimes n})
    = \bLp1 - \dtv(\mu_1, \mu_2)\bRp^n
    \ge \Lp1-\frac1n\Rp^n
    \ge \frac14~,
\end{align*}
hence
\begin{align*}
    \cR_n
    \ge \frac18 \metricratio{D_{M_1}(x,y)}{D_{M_2}(x,y)}
\end{align*}
which concludes the proof.
\end{proof}

Let us now prove \cref{thm:minimax_lower_bound} using \cref{lemma:lecam} with the measures $\mu_1$ and $\mu_2$ described in \cref{sec:minimax_lower_bound}.
The total variation distance between $\mu_1,\mu_2 \in \cM(d,L,\tau)$ is the fraction of volume of $M_1$ that was carved out to create $M_2$. The carved volume is upper bounded by $\|x-y\| \rho^{d-1}$ where $\rho$ is the distance of the middle of the black arcs in \cref{fig:carved_cube} to the missing edge, that is $\rho = (\tau-\delta) / \sqrt{d-1}$ where $\delta = \sqrt{\tau^2 - \epsilon^2}$ is the distance from the edge to the center of the ball used to carve $M_2$.
Then,
\begin{align*}
    \dtv(\mu_1,\mu_2)
    &\le \frac{\|x-y\| \rho^{d-1}}{(2\alpha L)^d}\\
    &= \frac{2\epsilon}{(\alpha L)^d (d-1)^{\frac{d-1}{2}}} \tau^{d-1} \Lp1-\sqrt{1 - \frac{\epsilon^2}{\tau^2}}\Rp^{d-1}\\
    &\le \frac{2}{\alpha^d (d-1)^{\frac{d-1}{2}}} \frac{\epsilon \tau^{d-1}}{L^d} \Lp\frac{\epsilon^2}{\tau^2}\Rp^{d-1}\\
    &= \beta \frac{\epsilon^{2d-1}}{L^d \tau^{d-1}}
\end{align*}
where $\beta = 2 \alpha^{-d} (d-1)^{-(d-1)/2}$ is a constant depending only on $d$.
Moreover, since a geodesic between $x$ and $y$ is a straight line in $M_1$ and an arc of radius $\tau$ in $M_2$, the induced distortion is
\begin{align*}
    \frac{D_{M_2}(x,y) - D_{M_1}(x,y)}{D_{M_2}(x,y) \vee D_{M_1}(x,y)}
    = \frac{\arcsin\Lp\frac{\|x-y\|}{2\tau}\Rp - \frac{\|x-y\|}{2\tau}}{\arcsin\Lp\frac{\|x-y\|}{2\tau}\Rp}
    \ge \frac16 \Lp\frac{\|x-y\|}{2\tau}\Rp^2
    = \frac{\epsilon^2}{6\tau^2}
\end{align*}
using \cref{eq:trigo_lower} to get the lower bound.
Then, by letting
\begin{align*}
    \epsilon
    = \Lp \frac{L^d \tau^{d-1}}{\beta~n} \Rp^{\frac{1}{2d-1}}~,
\end{align*}
which is indeed smaller than both $\alpha L$ and $\tau$ for $n$ large enough depending on $d$, $L$ and $\tau$, one has $\dtv(\mu_1,\mu_2) \le \frac1n$ and
\begin{align*}
    \frac{D_{M_2}(x,y) - D_{M_1}(x,y)}{D_{M_2}(x,y) \vee D_{M_1}(x,y)}
    \ge \frac16 \Lp \frac{L^d}{\beta~\tau^d~n} \Rp^{\frac{2}{2d-1}}~.
\end{align*}
\Cref{thm:minimax_lower_bound} then follows from \cref{lemma:lecam}.

\end{document}